\documentclass[12pt,a4paper,english]{article}
\usepackage[T1]{fontenc}
\usepackage[latin9]{inputenc}
\usepackage{amsmath}
\usepackage{amssymb}

\makeatletter

\providecommand{\tabularnewline}{\\}

\usepackage{amsthm}\usepackage{fullpage}\usepackage{makeidx}\usepackage{amsfonts}\usepackage{latexsym}\makeindex

\theoremstyle{definition}
\newtheorem{lemma}{Lemma}[section]
\newtheorem{theorem}[lemma]{Theorem}
\newtheorem{corollary}[lemma]{Corollary}
\newtheorem{proposition}[lemma]{Proposition}
\newtheorem{definition}[lemma]{Definition}

\usepackage{babel}

\makeatother

\usepackage{babel}

\usepackage[blocks]{authblk}% The option is for block layout
\title{Fusion rules for the vertex operator algebra
$V_{L_{2}}^{S_{4}}$}
\author{Junwen Liao}
\affil{Department of Mathematics, University of
California, Santa Cruz, CA 95064 USA}

\begin{document}
\maketitle

\begin{abstract}
This paper determines the fusion rules for the fixed-point vertex operator algebra $V_{L_2}^{S_4}$, where $L_2$ denotes the rank-one even lattice satisfying $(\alpha,\alpha)=2$. The main idea is to identify certain irreducible $V_{\mathbb Z\zeta}^{+}$-modules as direct sums of irreducible $V_{L_2}^{S_4}$-modules. These identifications allow us to compute the relevant entries of the $S$-matrix for $V_{L_2}^{S_4}$. Applying the Verlinde formula, we then obtain the complete fusion rules for $V_{L_2}^{S_4}$.
\end{abstract}

\section{Introduction }
\def\theequation{1.\arabic{equation}}
\setcounter{equation}{0}
  The classification of rational vertex operator algebras of central charge $1$ can be
traced back to the work of Kiritsis. In \cite{K}, Kiritsis showed that, under suitable
modularity assumptions, the $q$-character of a rational vertex operator algebra of
CFT type with central charge $c=\widetilde{c}=1$ coincides with the character of one
of the following vertex operator algebras:
$$
V_L,\qquad V_L^+,\qquad V_{\mathbb Z\alpha}^G,
$$
where $L$ is a positive-definite even lattice of rank one, $\mathbb Z\alpha$ is the
root lattice of type $A_1$, and $G$ is a finite subgroup of $SO(3)$ isomorphic to
$A_4$, $S_4$, or $A_5$.

The irreducible representations and fusion rules for $V_L$ and $V_L^+$ have been
studied by Frenkel--Lepowsky--Meurman, Abe--Dong, and Dong--Nagatomo. Subsequent
work has focused on the fixed-point vertex operator algebras
$$
V_{L_2}^{A_4},\qquad V_{L_2}^{S_4},\qquad V_{L_2}^{A_5},
$$
where $L_2$ is the rank-one lattice satisfying $(\alpha,\alpha)=2$. The irreducible
modules of $V_{L_2}^{A_4}$ were classified in \cite{DJ}, and the fusion rules for
$V_{L_2}^{A_4}$ were determined in \cite{DJJJY}. Furthermore, the irreducible modules
of $V_{L_2}^{S_4}$ were classified in \cite{WZ}.

The purpose of this paper is to determine the fusion rules for $V_{L_2}^{S_4}$.
Determining these fusion rules describes the representation category of $V_{L_2}^{S_4}$ once the
classification of its irreducible modules is known. The main strategy is to identify
certain irreducible $V_{\mathbb Z\zeta}^{+}$-modules as direct sums of irreducible
$V_{L_2}^{S_4}$-modules. These identifications allow us to compute the relevant
entries of the $S$-matrix for $V_{L_2}^{S_4}$. We then apply the Verlinde formula to
determine the complete fusion rules.

Two important tools used in this paper are the Verlinde formula and quantum
dimensions. The Verlinde formula provides a method for computing fusion rules from
the $S$-matrix of a rational, $C_2$-cofinite, self-dual vertex operator algebra of CFT
type. Quantum dimensions are used to control the possible decompositions of tensor
products and to verify that the decompositions obtained are complete.

This paper is organized as follows. In Chapter~2, we review the necessary
background on vertex operator algebras, twisted modules, lattice vertex operator
algebras, orbifold theory, intertwining operators, fusion rules, and quantum dimensions.
In Chapter~3, we recall the construction of the rank-one lattice vertex operator algebra
$V_{L_2}$, the fixed-point subalgebra $V_{L_2}^{S_4}$, and the classification of its
irreducible modules. We then establish the module identifications needed to compute
the relevant entries of the $S$-matrix. Finally, we use the Verlinde formula to determine
the fusion rules among all irreducible $V_{L_2}^{S_4}$-modules.

\section{Basics}
\def\theequation{2.\arabic{equation}}
\setcounter{equation}{0}

Let $(V, Y, \mathbf{1}, \omega)$ be a vertex operator algebra, and let $g$ be an automorphism of $V$ of finite order $T$. The eigenspace decomposition of $V$ with respect to $g$ is given by
\[
V = \bigoplus_{r=0}^{T-1} V^{r},
\]
where
\[
V^{r} = \{\, v \in V \mid g v = e^{\frac{2\pi i r}{T}} v \,\}.
\]

\begin{definition}A \emph{weak $g$-twisted $V$-module} $M$ is
a vector space with a linear map
\[
Y_{M}:V\to\left(\text{End}M\right)\{z\}
\]

\[
v\mapsto Y_{M}\left(v,z\right)=\sum_{n\in\mathbb{Q}}v_{n}z^{-n-1}\ \left(v_{n}\in\mbox{End}M\right)
\]
which satisfies the following: for all $0\le r\le T-1$, $u\in V^{r}$,
$v\in V$, $w\in M$,

\[
Y_{M}\left(u,z\right)=\sum_{n\in\frac{r}{T}+\mathbb{Z}}u_{n}z^{-n-1},
\]

\[
u_{l}w=0\ for\ l\gg0,
\]

\[
Y_{M}\left(\mathbf{1},z\right)=Id_{M},
\]

\[
z_{0}^{-1}\text{\ensuremath{\delta}}\left(\frac{z_{1}-z_{2}}{z_{0}}\right)Y_{M}\left(u,z_{1}\right)Y_{M}\left(v,z_{2}\right)-z_{0}^{-1}\delta\left(\frac{z_{2}-z_{1}}{-z_{0}}\right)Y_{M}\left(v,z_{2}\right)Y_{M}\left(u,z_{1}\right)
\]

\[
z_{2}^{-1}\left(\frac{z_{1}-z_{0}}{z_{2}}\right)^{-r/T}\delta\left(\frac{z_{1}-z_{0}}{z_{2}}\right)Y_{M}\left(Y\left(u,z_{0}\right)v,z_{2}\right),
\]
 where $\delta\left(z\right)=\sum_{n\in\mathbb{Z}}z^{n}$. \end{definition}

\begin{definition}

A $g$-\emph{twisted $V$-module} is a weak $g$-twisted $V$-module $M$ that carries a $\mathbb{C}$-grading induced by the spectrum of $L(0)$, where $L(0)$ is the component operator of
\[
Y(\omega,z)=\sum_{n\in\mathbb{Z}} L(n) z^{-n-2}.
\]
That is, we have
\[
M=\bigoplus_{\lambda\in\mathbb{C}} M_{\lambda},
\]
where
\[
M_{\lambda}=\{\, w\in M \mid L(0)w=\lambda w \,\}.
\]
Moreover, we require that $\dim M_{\lambda}$ is finite and that, for fixed $\lambda$, we have
\[
M_{\frac{n}{T}+\lambda}=0
\]
for all sufficiently small integers $n$.

\end{definition}

\begin{definition}An \emph{admissible $g$-twisted $V$-module} $M=\oplus_{n\in\frac{1}{T}\mathbb{Z}_{+}}M\left(n\right)$
is a $\frac{1}{T}\mathbb{Z}_{+}$-graded weak $g$-twisted module
such that $u_{m}M\left(n\right)\subset M\left(\mbox{wt}u-m-1+n\right)$
for homogeneous $u\in V$ and $m,n\in\frac{1}{T}\mathbb{Z}.$ $ $

\end{definition}
If $g=id_V$, these $g-$twisted notation become the notion of weak, ordinary, admissible module \cite{DLM1}.

\begin{definition} A vertex operator algebra $V$ is called \emph{$g$-rational} if any admissible $g$-twisted V module is a direct sum of irreducible $g$- twisted modules. $V$ is called rational if V is $1$-rational.

Let $M$ be an admissible $V$-module. Then the contragredient module
\[
M'=\bigoplus_{n\in \mathbb{Z}_{+}} M'(n),
\]
where
\[
M'(n)=\operatorname{Hom}_{\mathbb{C}}(M(n),\mathbb{C}).
\]
The vertex operator $Y_{M'}(v,z)$ is defined by
\begin{eqnarray*}
\langle Y_{M'}(v,z)f,u\rangle
= \langle f, Y_M\!\left(e^{zL(1)}(-z^{-2})^{L(0)}v, z^{-1}\right)u\rangle,
\end{eqnarray*}
for each $v\in V$, $u\in M$, and $f\in M'$.

Moreover, $M'$ is irreducible if and only if $M$ is irreducible, and $M'$ and $M$ have the same conformal weights.
 \begin{definition}
     A vertex operator algebra $V$ is called $C_2$-cofinite if $V/C_2(V)$ is finite-dimensional, where
\[
C_2(V)=\langle\, u_{(-2)}v \mid u,v \in V \,\rangle.
\]
 \end{definition}
    
\end{definition}
From \cite{DLM1}, we know if $V$ is rational, then it has only finitely many irreducible modules. 
\begin{definition}
    A vertex operator algebra $V$ is CFT type if $V_n=0$ for all negative n and $V_0=\mathbb{C}\mathbf{1}$.
\end{definition}

\section{Modular Invariance and Weyl decomposition}
  From \cite{DRX}, in the rest of this paper, we assume vertex operator algebra V  satisfies following requirement:
\begin{itemize}
    \item (V1) $V=\bigoplus_{n\geq 0} V_n$ is a simple vertex operator algebra of CFT type.

  \item (V2) G is a finite automorphism group of V and $V^G$ is a vertex operator algebra of CFT type.
\item (V3) $V^G$ is rational and $C_2$-cofinite.
\item (V4) The conformal weight of any g-twisted $V$ module for $g\in G$ is positive except $V$ itself.

\end{itemize}

If $V$ and $G$ satisfy above conditions, $V$ is $C_2$-cofinite and $g$-rational for all $g\in G$. \cite{ADJR,HKL,ABD}

  Let $V$ be a vertex operator algebra, $g$ is an automorphism of $V$ of order $T$, and $M$ = $\bigoplus_{n\in\frac{\mathbb{Z}_+}{T}}M_{\lambda+n}$ is a $g$-twisted $V$ module. The trace function associate to a homogeneous element $v\in V$ is :
  \begin{equation*}
      Z_M(v,q)= \mbox{tr}_M o(v)q^{L(0)-\frac{c}{24}}= q^{\lambda-\frac{c}{24}}\sum_{n\in \frac{\mathbb{Z}_+}{T}} \mbox{tr}_{M_{\lambda+n}} o(v)q^n.
  \end{equation*}
  where $o(v)=v_{wt(v)-1}$. If $|q|<1$ and $V$ is $C_2$-cofinite, then $Z_M(v,q)$ converges to a holomorphic function. In particular, if $v=\mathbf{1}$, then
\[
Z_M(\mathbf{1},q)
= q^{\lambda-\frac{c}{24}}
\sum_{n\in \frac{\mathbb{Z}_{+}}{T}} \dim M_{\lambda+n}\, q^{n}
\]
is called the character of $M$.
We will also use $Z_M(v,\tau)$ to denote $Z_M(v,q)$, where $\tau$ lies in the upper half-plane and $q=e^{2\pi i \tau}$.

By \cite{DLM2} and \cite{Z}, $Z_M(v,\tau)$ satisfies the modular invariance property:
  \begin{equation*}
      Z_{M_i}(v,\frac{-1}{\tau})=\tau^{wt[v]}\sum_{j=0}^{d}S_{i,j}Z_{M_j}(v,\tau)
  \end{equation*}
where $M_i$ is an irreducible V module and $M_j$ ranges over a complete list of irreducible V modules up to isomorphism. $S_{i,j}$ forms a matrix called \emph{$S$-matrix}, and this matrix is independent of choice of v.
  
 Let $g,h \in \operatorname{Aut}(V)$ be automorphisms of finite order. Let $(M,Y_M)$ be a $g$-twisted weak $V$-module. Then there exists an $h^{-1}gh$-twisted weak $V$-module $(M\circ h, Y_{M\circ h})$, where $M\circ h \cong M$ as vector spaces and
\[
Y_{M\circ h}(v,z)=Y_M(hv,z).
\]

If $g$ and $h$ commute, then $h$ acts on the set of $g$-twisted $V$-modules. Let $S(g)$ denote the set of equivalence classes of irreducible $g$-twisted $V$-modules, and define
\[
S(g,h)=\{\, M \in S(g) \mid M\circ h \cong M \text{ as $g$-twisted $V$-modules} \,\}.
\]
Then, for any $M \in S(g,h)$, there exists a $V$-module isomorphism
\[
\psi(h): M\circ h \rightarrow M.
\]
This linear map is unique up to a scalar, and when $h$ is the identity automorphism, $\psi(1)$ is the identity map.

 Let $M$ be a $g$-twisted $V$-module, and define
\[
G_M=\{\, h \in G \mid M\circ h \cong M \,\}.
\]
The assignment $h \mapsto \psi(h)$ defines a projective representation of $G_M$ on $M$ such that
\[
\psi(h)\, Y_M(v,z)\, \psi(h)^{-1} = Y_M(hv,z).
\]
Then $M$ is a $\mathbb{C}^{\alpha_M}[G_M]$-module, where $\mathbb{C}^{\alpha_M}[G_M]$ is a twisted group algebra and $\alpha_M$ is a $2$-cocycle.

Let $\Lambda_{G_M,\alpha_M}$ denote the set of irreducible characters of $\mathbb{C}^{\alpha_M}[G_M]$. Then we have
\[
M=\bigoplus_{\lambda \in \Lambda_{G_M, \alpha_M}} W_\lambda \otimes M_\lambda,
\]
where $W_\lambda$ is the irreducible $\mathbb{C}^{\alpha_M}[G_M]$-module corresponding to the irreducible character $\lambda$, and $M_\lambda$ is the corresponding multiplicity space.

    \begin{theorem} \cite{DRX} with the same notation above: 
        \begin{itemize}
            \item(1) $M_\lambda$ is nonzero for each $\lambda \in \Lambda_{G_M,\alpha_M}$.
            \item(2) Each $M_\lambda$ is an irreducible $V^{G_M}$-module.
            \item(3) $M_\lambda$ and $M_\gamma$ are equivalent $V^{G_M}$-module if and only if $\lambda = \gamma$
        \end{itemize}
    \end{theorem}
    Let $S=\bigcup_{g\in G} S(g)$. $G$ defines an action on $S$, and $M\cong M\circ h$ as $V^G$-module for any $M\in S$ and $h\in G$.
    \begin{theorem} \cite {DRX} Let $g,h\in G$, $M$ an irreducible $g$-twisted $V$-module and $N$ an irreducible $h$-twisted $V$-module. Assume $M,N$ are not in the same $G$-orbit, then:\begin{itemize}
        \item (1) Each $M_\lambda$ is an irreducible $V^G$ module.
        \item (2) For any $\lambda\in \Lambda_{G_M,\alpha_M}$ and $\mu \in \Lambda_{G_N,\alpha_N}$, the irreducible $V^G$ module $M_\lambda$ and $N_\mu$ are inequivalent.
    \end{itemize}
        
    \end{theorem}
    \begin{theorem}\cite{DRX}
        Any irreducible $V^G$-module is isomorphic to $M_\lambda$ for some $\lambda\in \Lambda_{G_M,\alpha_M}$ and some irreducible g-twisted module $M$.
    \end{theorem}

\section{Intertwining Operator and Fusion rule}
\begin{definition}
    Let $V$ be a vertex operator algebra and $(W_i,Y_{W_i})$ (i=1,2,3) be $V$ modules. An intertwining operator for $V$ of type $\left(\begin{array}{c}W_3\\
W_{1}\,W_{2}\end{array}\right)$ is a linear map $I$: $M_{1}\otimes M_{2} \rightarrow M_{3}\{z\}$ such that for $a\in V,v\in M_{1}$ and $u\in M_{2}$, following conditions are satisfied:

For fixed $n\in \mathbb{C},v_{n+k}u=0$ for sufficiently large integer $k$,

\[z_0^{-1}\delta(\frac{z_1-z_2}{z_0})Y_{M_3}(a,z_1)I(v,z_2)u-z_0^{-1}\delta(\frac{z_2-z_1}{-z_0})I(v,z_2)Y_{M_2}(a,z_1)u=z_1^{-1}\delta(\frac{z_1-z_0}{z_2})I(Y_{M_1}(a,z_0)v,z_2)u,\]

\[\frac{d}{dz}I(v,z)=I(L(-1)v,z).\]
\end{definition}
Let $I_V\left(\begin{array}{c}W_3\\
W_{1}\,W_{2}\end{array}\right)$ be the vector space of all intertwining operators of type $\left(\begin{array}{c}W_3\\
W_{1}\,W_{2}\end{array}\right)$. The fusion rule $N_{W_{1},\ W_{2}}^{W_{3}}$ is the dimension of the vector space $I_V\left(\begin{array}{c}W_3\\
W_{1}\,W_{2}\end{array}\right)$. From \cite{FHL}, we know $I_V\left(\begin{array}{c}W_3\\
W_{1}\,W_{2}\end{array}\right)\cong I_V\left(\begin{array}{c}W_3\\
W_{2}\,W_{1}\end{array}\right)\cong I_V\left(\begin{array}{c}(W_2)'\\
W_{1}\,(W_{3})'\end{array}\right)$.
\begin{definition}
    Let $V$ be a vertex operator algebra, and $W_{1},W_{2}$ be two $V$-modules. A module $(W,I)$, where $I\in I_V\left(\begin{array}{c}W\\
W_{1}\,W_{2}\end{array}\right)$, is called a tensor product of $W_{1}$ and $W_{2}$ if for any $V$-module homomorphism $f:W\rightarrow M$, such that $y=f\circ I$. We denote $(W,I)$ by $W_{1}\boxtimes_{V} W_{2}$. 
\end{definition}

The tensor product exists if $V$ is rational. When $W^{1}$ and $W^{2}$ are two irreducible $V$-modules, we have
\[
W_{1} \boxtimes_{V} W_{2}
= \sum_{W} N_{W_{1},\, W_{2}}^{W}\, W,
\]
where $W$ runs over the set of equivalence classes of irreducible $V$-modules. From \cite{H} and \cite{V}, we have the following theorem:

\begin{theorem}(Verlinde formula) \label{Verlinde formula} Let $V$ be a rational
and $C_{2}$-cofinite simple vertex operator algebra of CFT type and
assume $V\cong V'$. Let $S=\left(S_{i,j}\right)_{i,j=0}^{d}$ be
the $S$-matrix as defined above. Then

(1) $\left(S^{-1}\right)_{i,j}=S_{i,j'}=S_{i',j}$, and $S_{i',j'}=S_{i,j};$

(2) $S$ is symmetric and $S^{2}=\left(\delta_{i,j'}\right)$;

(3) $N_{i,j}^{k}=\sum_{s=0}^{d}\frac{S_{i,s}S_{j,s}S_{s,k}^{-1}}{S_{0,s}}$.

\end{theorem}
This following theorem also plays an important role in computing fusion rules for $V_{L_2}^{S_4}$\cite{DL}, \cite{A}:  
\begin{theorem}
    Let $V$ be a vertex operator algebra, $M_1,M_2$ be irreducible $V$-module and $M_3$ be a $V$-module. Let $U$ be a subalgebra of $V$ with same virasoro element and $N_i$ be U-module in $M_i$, ($i=1,2$). We have:
    \[\mbox{dim }I_V\left(\begin{array}{c}M_3\\
M_{1}\,M_{2}\end{array}\right) \leq \mbox{dim }I_U\left(\begin{array}{c}M_3\\
N_{1}\,N_{2}\end{array}\right).\]
\end{theorem}

\section{Quantum Dimension}
Quantum dimensions also play an important role in our calculation of fusion rules.
 
\begin{definition}
    Let $V$ be a vertex operator algebra and $M$ a $V$-module such that $Z_V(\tau)$ and $Z_M(\tau)$ exist. The quantum dimensions of $M$ over $V$ is defined as 
    \[\mbox{qdim}_{V}M=\lim_{y\rightarrow 0}\frac{Z_{M}(iy)}{Z_V{(iy)}}\]
    where $y$ is real and positive. \cite{DJX}
\end{definition}
From Remark 3.5 of \cite{DJX}, if $\operatorname{qdim}_V M$ exists, then
\[
\operatorname{qdim}_V M = \operatorname{qdim}_V M'.
\]
In the calculation of Lemma~4.2 in \cite{DJX}, if $V$ is a simple, rational, $C_2$-cofinite vertex operator algebra of CFT type and the conformal weights of all irreducible $V$-modules are positive except that of $V$ itself, then
\[
\operatorname{qdim}_V M^i = \frac{S_{i,0}}{S_{0,0}},
\]
where $M^i$ is an irreducible $V$-module and $S_{i,j}$ denotes the $(i,j)$-entry of the $S$-matrix. By convention, we let $M^0 = V$.

\begin{theorem}
    By \cite{DJX}, let $V$ be a simple, rational, $C_2$-cofinite vertex operator algebra of CFT type. Let $M_0, M_1, \dots, M_d$ be a complete list of equivalence classes of irreducible $V$-modules such that $M_0 = V$, and the conformal weights of all irreducible $V$-modules except $M_0$ are positive. Also, assume that $V \cong V'$ as a $V$-module. Then,
 
    \[\mbox{qdim}_V(M_i\boxtimes M_j)=\mbox{qdim}_V(M_i)\mbox{qdim}_V{M_j}\]
\end{theorem}
\begin{definition}
    Let $V$ be a simple vertex operator algebra. An irreducible $V$-module $M$ is called a simple current if for any irreducible $V$-module $W$, the tensor product $M\boxtimes W$ exists, and $M\boxtimes W$ is an irreducible $V$-module.
\end{definition}
\begin{theorem}
    By \cite{DJX}, let $V$ be a vertex operator algebra as in theorem 5.2. Then $M$ is a simple current if and only if $\mbox{qdim}_V(M)=1$.
\end{theorem}
\begin{definition}
    The global dimension of $V$ is defined as $\mbox{glob}(V)=\sum_{i=0}^{d} \mbox{qdim}M_i^2$.
\end{definition}
\begin{theorem}
    By \cite{DRX}, we have: 
    \[\mbox{qdim}_{V^G}M=|G|\mbox{qdim}_VM.\] for any irreducible $g$-twisted $V$-module $M$.
\end{theorem}
\begin{theorem}
    By \cite{DRX}, the following holds:
    \[\mbox{glob}(V^G)=|G|^2\mbox{glob}(V).\]
\end{theorem}
\section{Irreducible modules of $V_{L_2}^{S_4}$}
\setcounter{equation}{0}
\numberwithin{equation} {section}
Let $L=\mathbb{Z}\alpha$ be a positive definite even lattice of rank one,$(\alpha,\alpha)=2k$. Let $\mathfrak{h} = L\otimes_{\mathbb{Z}}\mathbb{C}$ and extend $(.,.)$ to a $\mathbb{C}$-bilinear form on $\mathfrak{h}$. Let $\mathbb{C}[\mathfrak{h}]$ be the group algebra of $\mathfrak{h}$ with a basis $\{e^{\lambda}|\lambda \in \mathfrak{h}\}$.

Let $\hat{\mathfrak{h}} = \mathbb{C}[t,t^{-1}] \otimes \mathfrak{h} \oplus \mathbb{C}K$ be the corresponding Heisenberg algebra, and let $\hat{\mathfrak{h}}_{\geq 0} = \mathbb{C}[t] \otimes \mathfrak{h} \oplus \mathbb{C}K$ be a subalgebra of $\hat{\mathfrak{h}}$. The space $\mathbb{C}[\mathfrak{h}]$ has a $\hat{\mathfrak{h}}_{\geq 0}$-module structure defined by
\[
\alpha(m) e^{\lambda} = (\lambda, \alpha) \delta_{m,0} e^{\lambda}, \quad K e^{\lambda} = e^{\lambda},
\]
for any $\lambda \in \mathfrak{h}$ and $m \ge 0$, where $\alpha(m) = \alpha \otimes t^m$ for $m \in \mathbb{Z}$. Denote by $M(1,\lambda)$ the induced $\hat{\mathfrak{h}}$-module of $\mathbb{C} e^{\lambda}$, and set $M(1) = M(1,0)$. Let $\mathbb{C}[L]$ be the group algebra of $L$ with basis $\{e^{\alpha} \mid \alpha \in L\}$. The lattice vertex operator algebra associated to $L$ is defined as
\[
V_L = M(1) \otimes \mathbb{C}[L].
\]

Let $L^{\circ}$ be the dual lattice of $L$:
\[
L^{\circ} = \{\lambda \in \mathfrak{h} \mid (\alpha, \lambda) \in \mathbb{Z} \} = \frac{1}{2k} L.
\]
Then $L^{\circ}$ admits the coset decomposition
\[
L^{\circ} = \bigcup_{i=-k+1}^{k} (L + \lambda_i), \quad \text{where } \lambda_i = \frac{i}{2k} \alpha.
\]
Set $V_{L+\lambda_i} = M(1) \otimes \mathbb{C}[L+\lambda_i]$. The modules $V_{L+\lambda_i}$, for $i=-k+1, \dots, k$, are all inequivalent irreducible $V_L$-modules.

From \cite{FLM}, there is an automorphism $\theta$ of $V_{L}$:
\[\theta(\alpha(-n_1)\cdots \alpha(-n_k)\otimes e^{\lambda})=(-1)^k\alpha(-n_1)\cdots \alpha(-n_k)\otimes e^{-\lambda}.\] for $n\in \mathbb{Z}_{+}$ and $\lambda\in \mathfrak{h}$. Let $W$ be a invariant space of $\theta$ and denote the $\pm1$ eigenspace by $W^{\pm}$. In particular, $V_L^{+}$ is a vertex operator algebra, and $V_{L}^-$ is an irreducible $V_L^+$-module. Two $\theta$-twisted $V_L$-modules are constructed in the following way:

Let $\hat{\mathfrak{h}}[-1]=\mathfrak{h}\otimes t^{\frac{1}{2}}\mathbb{C}[t,t^{-1}]\oplus \mathbb{C}K$ be a Lie algebra and $\hat{\mathfrak{h}}[-1]_+=\mathfrak{h}\otimes t^{\frac{1}{2}}\mathbb{C}[t]\oplus \mathbb{C}K$ is a subalgebra of $\hat{\mathfrak{h}}[-1]$. Then, $\mathbb{C}$ can be considered as a $1$-dimensional module for $\hat{\mathfrak{h}}[-1]_+$ with the action 
\[(\alpha \otimes t^m)\cdot 1=0 \mbox{ and }K\cdot 1=1 \mbox{ for } m\in \frac{1}{2}+\mathbb{N}.\]
Let $M(1)(\theta)$ be the induced $\hat{\mathfrak{h}}[-1]$-module. Let $\chi_s$ be a character of $L/2L$ such that $\chi_s(\alpha) = (-1)^s$ for $s = 0,1$, and let $T_{\chi_s} = \mathbb{C}$ be the irreducible $L/2L$-module with character $\chi_s$. Then 
\[
V_L^{T_s} = M(1)(\theta) \otimes T_{\chi_s}
\] 
is an irreducible $\theta$-twisted $V_L$-module. We denote the $\pm 1$-eigenspaces of $V_L^{T_s}$ under $\theta$ by $(V_L^{T_s})^{\pm}$. To agree with the notation in \cite{WZ}, in the rest of the paper we will denote $V_L^{T_0}$ by $V_L^{T_1}$ and $V_L^{T_1}$ by $V_L^{T_2}$.

From \cite{DN}, all inequivalent $V_L^+$ modules are classified 
\begin{theorem}
    Any irreducible $V_L^+$- module is isomorphic to one of those modules:
    \[V_L^{\pm}, V_{\lambda_i+L}(i\neq k), V_{\lambda_k+L}^{\pm},(V_L^{T_1})^{\pm},(V_L^{T_2})^{\pm}.\]
\end{theorem}
Let $L_2$ be the rank 1 positive definite even lattice with $(\alpha,\alpha)=2$. Then $x_1,x_2,x_3$ form an orthonormal basis of $(V_{L_2})_1$
\[x_1=\frac{1}{\sqrt{2}}\alpha(-1)\mathbf{1},x_2=\frac{1}{\sqrt{2}}(e^\alpha+e^{-\alpha}),x_3=\frac{i}{\sqrt{2}}(e^\alpha-e^{-\alpha}).\]
Let $\delta,\tau_i, \rho\in \mbox{Aut}(V_{L_2}), i = 1,2$ be the following 
\[\delta\left(x_{1},x_{2},x_{3}\right)=\left(x_{1},x_{2},x_{3}\right)\left[\begin{array}{ccc}
0 & 1 & 0\\
0 & 0 & -1\\
-1 & 0 & 0
\end{array}\right].
\]
\[
\tau_{1}\left(x_{1},x_{2},x_{3}\right)=\left(x_{1},x_{2},x_{3}\right)\left[\begin{array}{ccc}
1\\
 & -1\\
 &  & -1
\end{array}\right],
\]
\[
\tau_{2}\left(x_{1},x_{2},x_{3}\right)=\left(x_{1},x_{2},x_{3}\right)\left[\begin{array}{ccc}
-1\\
 & 1\\
 &  & -1
\end{array}\right],
\]
\[ \rho \left(x_{1},x_{2},x_{3}\right)=\left(x_{1},x_{2},x_{3}\right)\left[\begin{array}{ccc}
-1 & 0 & 0\\
0 & 0 & 1\\
0 & 1 & 0
\end{array}\right].
\]
It is easy to see that, for $i=1,2$, 
\[
\langle \tau_i, \delta, \rho \rangle \cong S_4, \quad 
\langle \tau_i, \delta \rangle \cong A_4, \quad 
\langle \tau_i \rangle \cong K_4,
\] 
and the $K_4$ generated by $\tau_i$ is a normal subgroup of both $S_4$ and $A_4$.
\begin{lemma}
    By \cite{DG},\cite{WZ}, we have $V_{L_2}^{K}=V_{\mathbb{Z}\beta}^+, V_{L_2}^{A_4}=(V_{\mathbb{Z}\beta}^+)^{<\delta>}, V_{L_2}^{S_4}=(V_{\mathbb{Z}\beta}^+)^{<\delta,\rho>}=(V_{L_2}^{A_4})^{\rho}$.
\end{lemma}
    By the remark 3.4 of \cite{DJJJY}, $SO(3)$ is the connected compact subgroup of $\mbox{Aut}(V_{L_2})$ that contains $Z_n,D_n,A_4,S_4, A_5$ as discrete subgroups. We can get $V_{L_2}^{Z_n}\cong V_{\mathbb{Z}n\alpha}$ and $V_{L_2}^{D_n}\cong V_{\mathbb{Z}n\alpha}^{+}$.

let us denote $\beta=2\alpha,\zeta=4\alpha$. By \cite{DG}, $V_{\mathbb{Z}\beta}^+$ is generated by $J=h(-1)^4\mathbf{1} -2h(-3)h(-1)\mathbf{1}+\frac{3}{2}h(-2)^2\mathbf{1},E_{\beta}=e^\beta+e^{-\beta}$ and $V_{\mathbb{Z}\zeta}^+$ is generated by $J,E_\zeta=e^{\zeta}+e^{-\zeta}$ where $h=\frac{1}{\sqrt{2}}\alpha$. By \cite{WZ}, $V_{L_2}^{S_4}$ is rational and $C_2$- cofinite, and all inequivalent irreducible $V_{L_2}^{S_4}$-modules are classified as follows and their quantum dimensions are calculated.
\begin{theorem}
    Any irreducible $V_{L_2}^{S_4}$- module is isomorphic to one of those modules:
    \begin{equation}
        ((V_{\mathbb{Z}\beta}^+)^0)^+,((V_{\mathbb{Z}\beta}^+)^0)^-,(V_\mathbb{Z\beta}^+)^1,(V_{\mathbb{Z}\beta}^-)^+,(V_{\mathbb{Z}\beta}^-)^-,
    \end{equation}
    \begin{eqnarray}
        (V_{\mathbb{Z\beta}+\frac{1}{4}\beta}^0)^+,(V_{\mathbb{Z\beta}+\frac{1}{4}\beta}^0)^-,V_{\mathbb{Z\beta}+\frac{1}{4}\beta}^1,
    \end{eqnarray}
    \begin{eqnarray}
        (V_{\mathbb{Z}\beta+\frac{1}{8}\beta})^+,(V_{\mathbb{Z}\beta+\frac{1}{8}\beta})^-,(V_{\mathbb{Z}\beta+\frac{3}{8}\beta})^+,(V_{\mathbb{Z}\beta+\frac{3}{8}\beta})^-,
    \end{eqnarray}
    \begin{eqnarray}
        V_{\mathbb{Z}\gamma+\frac{r}{18}\gamma},\mbox{for } 1\leq r \leq 8 \mbox{ and } r\neq 0\mbox{ (mod }3),
    \end{eqnarray}
    \begin{eqnarray}
        V_{\mathbb{Z}\zeta+\frac{s}{32}\zeta},\mbox{for } 1\leq s \leq 15 \mbox{ and } s\neq 0\mbox{ (mod }2),
    \end{eqnarray}
    \begin{eqnarray}
        V_{\mathbb{Z}\zeta}^{T_2,+},V_{\mathbb{Z}\zeta}^{T_2,-}
    \end{eqnarray}
\end{theorem}
\begin{theorem}
     The quantum dimensions for all irreducible $V_{L_2}^{S_4}$-modules over $V_{L_2}^{S_4}$
are given by the following tables.
\begin{center}
\begin{tabular}{|c|c|c|c|c|c|}
\hline
 & \(((V^+_{\mathbb{Z}\beta})^0)^+\) & \(((V^+_{\mathbb{Z}\beta})^0)^{-}\) & \((V^+_{\mathbb{Z}\beta})^1\)& \((V^-_{\mathbb{Z}\beta})^{+}\)& \((V^-_{\mathbb{Z}\beta})^{-}\)\tabularnewline
\hline
$\mathrm{qdim}$ & 1 & 1 & 2 & 3 & 3\tabularnewline
\hline
$$ & \(M^{0}\) & \(M^{1}\) & \(M^{2}\) & \(M^{3}\)& \(M^{4}\)\tabularnewline
\hline
\end{tabular}
\par\end{center}

\begin{center}
\begin{tabular}{|c|c|c|c|}
\hline
 & \((V_{\mathbb{Z}\beta+\frac{1}{4} \beta }^0)^+\) & \((V_{\mathbb{Z}\beta+\frac{1}{4} \beta }^0)^{-}\) & \(V_{\mathbb{Z}\beta+\frac{1}{4} \beta }^1\)\tabularnewline
\hline
$\mathrm{qdim}$ & 2 & 2 & 4\tabularnewline
\hline
$$ & \(M^{5}\) & \(M^{6}\)& \(M^{7}\)\tabularnewline
\hline
\end{tabular}
\par\end{center}

\begin{center}
\begin{tabular}{|c|c|c|c|c|}
\hline
 & \((V_{\mathbb{Z}\beta+\frac{1}{8} \beta })^+\) & \((V_{\mathbb{Z}\beta+\frac{1}{8} \beta })^-\) & \((V_{\mathbb{Z}\beta+\frac{3}{8} \beta })^+\)& \((V_{\mathbb{Z}\beta+\frac{3}{8} \beta })^{-}\)\tabularnewline
\hline
$\mathrm{qdim}$ & 6 & 6 & 6 & 6\tabularnewline
\hline
$$ & \(M^{8}\) & \(M^{9}\)& \(M^{10}\) & \(M^{11}\) \tabularnewline
\hline
\end{tabular}
\par\end{center}

\begin{center}
\begin{tabular}{|c|c|c|c|c|}
\hline
 & \(V_{\mathbb{Z}\gamma+\frac{r}{18}\gamma}\) & \(V_{\mathbb{Z}\zeta+\frac{s}{32}\zeta}\) & \(V_{\mathbb{Z}\zeta}^{T_2,+}\)& \(V_{\mathbb{Z}\zeta}^{T_2,-}\)\tabularnewline
\hline
$\mathrm{qdim}$ & 8 & 6 & 12 & 12\tabularnewline
\hline
$$ & \(M^{12},\ldots, M^{17}\) & \(M^{18}, \ldots, M^{25}\)& \(M^{26}\) & \(M^{27}\) \tabularnewline
\hline
\end{tabular}
\par\end{center}
In the last table, \(r,s\in \mathbb{Z}\), \(1\leq r\leq 8\), \(1\leq s\leq 15\), \(r\neq 0\) (mod 3), and \(s\neq 0\) (mod 2).
\end{theorem}
Let $P = \rho \circ \tau_2$ and note that 
\[
\langle P, \tau_2 \rangle = \langle \tau_1, \tau_2, \rho \rangle \cong D_8.
\] 
By direct computation, $P$ maps $\alpha(-1)$ to $\alpha(-1)$, $e^\alpha$ to $i e^\alpha$, and $e^{-\alpha}$ to $-i e^{-\alpha}$, so $P$ fixes $e^{\zeta}$ and $\zeta(-1)$. Therefore, $V_{L_2}^{\langle P \rangle} = V_{\mathbb{Z}\zeta}$, not just up to isomorphism. On the other hand, $\tau_2$ is exactly the automorphism $\theta$ discussed above, so
\[
(V_{L_2}^{\langle P \rangle})^{\langle \tau_2 \rangle} = (V_{L_2})^{\langle P, \tau_2 \rangle} = V_{\mathbb{Z}\zeta}^+ = (V_{\mathbb{Z}\beta}^+)^{\langle \rho \rangle}.
\] 
With some abuse of notation, here $M^{\pm}$ will denote the $\pm 1$-eigenspace with respect to $\psi(\rho)$ if $M$ is a $V_{\mathbb{Z}\beta}^+$-module.

\begin{proposition}
    As $V_{L_2}^{S_4}$-modules, we have the following identification:
    \begin{equation}
        V_{\mathbb{Z}\beta}^{++}\cong V_{\mathbb{Z}\zeta}^{+}\cong ((V_{\mathbb{Z}\beta}^+)^0)^+ \oplus (V_{\mathbb{Z}\beta}^+)^1\label{6.1}
    \end{equation}
    \begin{equation}
        V_{\mathbb{Z}\beta}^{+-}\cong V_{\mathbb{Z}\zeta+\frac{1}{2}\zeta}^+\cong ((V_{\mathbb{Z}\beta}^+)^0)^- \oplus (V_{\mathbb{Z}\beta}^+)^1\label{6.2}
    \end{equation}
    
    \begin{equation}
        V_{\mathbb{Z}\beta}^{-+}\cong V_{\mathbb{Z}\zeta}^-\label{6.3}
    \end{equation}
    \begin{equation}
        V_{\mathbb{Z}\beta}^{--}\cong V_{\mathbb{Z}\zeta+\frac{1}{2}\zeta}^-\label{6.4}
    \end{equation}
    \begin{equation}
        V_{\frac{1}{8}\beta+\mathbb{Z}\beta}^+\cong V_{\frac{2}{32}\zeta+\mathbb{Z}\zeta}\label{6.5}
    \end{equation}
    \begin{equation}
        V_{\frac{3}{8}\beta+\mathbb{Z}\beta}^+\cong V_{\frac{6}{32}\zeta+\mathbb{Z}\zeta}\label{6.6}
    \end{equation}
    \begin{equation}
        V_{\frac{1}{8}\beta+\mathbb{Z}\beta}^-\cong V_{\frac{14}{32}\zeta+\mathbb{Z}\zeta}\label{6.7}
    \end{equation}
    \begin{equation}
        V_{\frac{3}{8}\beta+\mathbb{Z}\beta}^-\cong V_{\frac{10}{32}\zeta+\mathbb{Z}\zeta}\label{6.8}
    \end{equation}
   \begin{equation}
       V_{\frac{1}{4}\beta+\mathbb{Z}\beta}^+\cong V_{\frac{4}{32}\zeta+\mathbb{Z}\zeta}\cong (V_{\frac{1}{4}\beta +\mathbb{Z}\beta}^0)^+ \oplus V_{\frac{1}{4}\beta +\mathbb{Z}\beta}^1\label{6.9}
   \end{equation}
   \begin{equation}
       V_{\frac{1}{4}\beta+\mathbb{Z}\beta}^-\cong V_{\frac{12}{32}\zeta+\mathbb{Z}\zeta}\cong (V_{\frac{1}{4}\beta +\mathbb{Z}\beta}^0)^- \oplus V_{\frac{1}{4}\beta +\mathbb{Z}\beta}^1\label{6.10}
   \end{equation}
   \begin{equation}
       V_{\mathbb{Z}\zeta}^{T_1,+}\cong (V_{\mathbb{Z}\beta+\frac{\beta}{8}})^+\oplus(V_{\mathbb{Z}\beta+\frac{\beta}{8}})^-\label{6.11}
   \end{equation}
   \begin{equation}
       V_{\mathbb{Z}\zeta}^{T_1,-}\cong (V_{\mathbb{Z}\beta+\frac{3\beta}{8}})^+\oplus(V_{\mathbb{Z}\beta+\frac{3\beta}{8}})^-\label{6.12}
   \end{equation}
   \begin{equation}
       V_{\mathbb{Z}\beta+\frac{1}{2}\beta}^+\cong V_{\mathbb{Z}\beta+\frac{\beta}{2}}^-\cong V_{\frac{8\zeta}{32}+\mathbb{Z}\zeta} \label{1}
   \end{equation}
    
\end{proposition}
\begin{proof}
    First, we prove (\ref{6.5})--(\ref{6.8}). By the definition of $\rho$, we have $\rho(J)=J$ and $\rho(E_{\beta})=-E_{\beta}$. On the one hand, by \cite{WZ}, we have
\[
V_{\frac{\beta}{8}+\mathbb{Z}\beta}
\cong
V_{\frac{\beta}{8}+\mathbb{Z}\beta}^+
\oplus
V_{\frac{\beta}{8}+\mathbb{Z}\beta}^-
\]
as a $V_{L_2}^{S_4}$-module. On the other hand, applying Theorem~3.1 to decompose $V_{\frac{\beta}{8}+\mathbb{Z}\beta}$ into $V_{\mathbb{Z}\zeta}^+$-modules, we again obtain
\[
V_{\frac{\beta}{8}+\mathbb{Z}\beta}
\cong
V_{\frac{\beta}{8}+\mathbb{Z}\beta}^+
\oplus
V_{\frac{\beta}{8}+\mathbb{Z}\beta}^-
\]
as a $V_{\mathbb{Z}\zeta}^+$-module. Thus, $V_{\frac{\beta}{8}+\mathbb{Z}\beta}^+$ and $V_{\frac{\beta}{8}+\mathbb{Z}\beta}^-$ must be isomorphic to irreducible $V_{\mathbb{Z}\zeta}^+$-modules listed in Theorem~6.1.

Without loss of generality, we may assume that $\psi(\rho)(e^{\frac{\beta}{8}})=e^{\frac{\beta}{8}}$. Then the conformal weight of $V_{\frac{\beta}{8}+\mathbb{Z}\beta}^+$ is $\frac{1}{16}$. By the definition of $\psi(\rho)$, we have
\[
\psi(\rho)Y(E_{\beta},z)=-Y(E_{\beta},z)\psi(\rho).
\]
Hence,
\[
(E_{\beta})_{(n)}e^{\frac{\beta}{8}} \in V_{\frac{\beta}{8}+\mathbb{Z}\beta}^-
\]
for every $n$ such that $(E_{\beta})_{(n)}e^{\frac{\beta}{8}}\neq 0$. By direct calculation,
\[
Y(E_{\beta},z)e^{\frac{\beta}{8}}
=
z E^{-}(-\beta,z)e^{\frac{9\beta}{8}}
+
z^{-1}E^{-}(\beta,z)e^{-\frac{7\beta}{8}}.
\]
Comparing with conformal weights of irreducible $V_{\mathbb{Z}\zeta}^+$-modules, it follows that the conformal weight of $V_{\frac{\beta}{8}+\mathbb{Z}\beta}^-$ is $\frac{49}{16}$. we obtain (\ref{6.5}) and (\ref{6.7}). Applying the same argument yields (\ref{6.6}) and (\ref{6.8}).

Next, we prove (\ref{6.9}) and (\ref{6.10}). Using a similar method, we find that $V_{\frac{\beta}{4}+\mathbb{Z}\beta}^+$ has conformal weight $\frac{1}{4}$ and $V_{\frac{\beta}{4}+\mathbb{Z}\beta}^-$ has conformal weight $\frac{9}{4}$. Hence,
\[
V_{\frac{\beta}{4}+\mathbb{Z}\beta}^+
\cong
V_{\frac{4}{32}\zeta+\mathbb{Z}\zeta},
\qquad
V_{\frac{\beta}{4}+\mathbb{Z}\beta}^-
\cong
V_{\frac{12}{32}\zeta+\mathbb{Z}\zeta}
\]
as $V_{\mathbb{Z}\zeta}^+$-modules. The corresponding decomposition as $V_{L_2}^{S_4}$-modules follows from Theorem~6.4.

By \cite{DN}, we have
\[
V_{\mathbb{Z}\beta+\frac{\beta}{2}}^+
\cong
V_{\mathbb{Z}\beta+\frac{\beta}{2}}^-
\]
as $V_{\mathbb{Z}\zeta}^+$-modules, and both are isomorphic to $V_{\frac{8}{32}\zeta+\mathbb{Z}\zeta}$. By \cite{WZ},
\[
V_{\mathbb{Z}\beta}^{T_1,+}
\cong
V_{\mathbb{Z}\beta}^{T_2,+}
\cong
V_{\mathbb{Z}\zeta}^{T_1,+},
\qquad
V_{\mathbb{Z}\beta}^{T_1,-}
\cong
V_{\mathbb{Z}\beta}^{T_2,-}
\cong
V_{\mathbb{Z}\zeta}^{T_1,-}
\]
as $V_{\mathbb{Z}\zeta}^+$-modules. On the other hand, by \cite{DJ},
\[
V_{\mathbb{Z}\beta+\frac{\beta}{8}}
\cong
V_{\mathbb{Z}\beta}^{T_1,+}
\cong
V_{\mathbb{Z}\beta}^{T_2,+},
\qquad
V_{\mathbb{Z}\beta+\frac{3\beta}{8}}
\cong
V_{\mathbb{Z}\beta}^{T_1,-}
\cong
V_{\mathbb{Z}\beta}^{T_2,-}
\]
as $V_{L_2}^{A_4}$-modules. Thus, (\ref{6.11}) and (\ref{6.12}) follow.

Comparing with the remaining inequivalent $V_{\mathbb{Z}\zeta}^+$-modules, it is clear that
\[
V_{\mathbb{Z}\beta}^{++} \cong V_{\mathbb{Z}\zeta}^+,
\qquad
V_{\mathbb{Z}\beta}^{-+} \cong V_{\mathbb{Z}\zeta}^-.
\]
By an argument similar to that used to prove (\ref{6.5}), we see that $V_{\mathbb{Z}\beta}^{+-}$ is generated by $E_{\beta}$. A direct computation shows that
\[
(E_{\zeta})_{(15)}(E_{\beta}) = E_{\beta}.
\]
Therefore, by \cite{DN},
\[
V_{\mathbb{Z}\beta}^{+-}
\cong
V_{\mathbb{Z}\zeta+\frac{1}{2}\zeta}^+,
\qquad
V_{\mathbb{Z}\beta}^{--}
\cong
V_{\mathbb{Z}\zeta+\frac{1}{2}\zeta}^-
\]
as $V_{\mathbb{Z}\zeta}^+$-modules and hence as $V_{L_2}^{S_4}$-modules. Finally, (\ref{6.1}) and (\ref{6.2}) follow from quantum dimensions.

\end{proof}
\begin{corollary}
    \begin{equation}
        V_{\mathbb{Z}\zeta}\cong ((V_{\mathbb{Z}\beta}^+)^0)^+ \oplus (V_{\mathbb{Z}\beta}^+)^1 \oplus (V_{\mathbb{Z}\beta}^-)^+\label{6.13}
    \end{equation}
    \begin{equation}
        V_{\mathbb{Z}\zeta+\frac{\zeta}{4}}\cong(V_{\mathbb{Z}\beta}^-)^+\oplus (V_{\mathbb{Z}\beta}^-)^- \label{6.14}
    \end{equation}
    \begin{equation}
        V_{\mathbb{Z}\zeta+\frac{\zeta}{2}}\cong ((V_{\mathbb{Z}\beta}^+)^0)^- \oplus (V_{\mathbb{Z}\beta}^+)^1\oplus (V_{\mathbb{Z}\beta}^-)^-\label{6.15}
    \end{equation}
    \begin{equation}
        V_{\mathbb{Z}\zeta+\frac{3\zeta}{4}}\cong(V_{\mathbb{Z}\beta}^-)^+\oplus (V_{\mathbb{Z}\beta}^-)^- \label{6.16}
    \end{equation}
    \begin{equation}
        V_{\mathbb{Z}\zeta+\frac{\zeta}{8}\cong (V_{\mathbb{Z}\beta+\frac{\beta}{4}}^0)^+\oplus (V_{\mathbb{Z}\beta+\frac{\beta}{4}}^1)} \label{6.17}
    \end{equation}
    \begin{equation}
        V_{\mathbb{Z}\zeta+\frac{3\zeta}{8}\cong (V_{\mathbb{Z}\beta+\frac{\beta}{4}}^0)^-\oplus (V_{\mathbb{Z}\beta+\frac{\beta}{4}}^1)} \label{6.18}
    \end{equation}
    \begin{equation}
        V_{\mathbb{Z}\zeta+\frac{5\zeta}{8}\cong (V_{\mathbb{Z}\beta+\frac{\beta}{4}}^0)^-\oplus (V_{\mathbb{Z}\beta+\frac{\beta}{4}}^1)} \label{6.19}
    \end{equation}
    \begin{equation}
        V_{\mathbb{Z}\zeta+\frac{7\zeta}{8}\cong (V_{\mathbb{Z}\beta+\frac{\beta}{4}}^0)^+\oplus (V_{\mathbb{Z}\beta+\frac{\beta}{4}}^1)} \label{6.20}
    \end{equation}
\end{corollary}
(6.21) follows from lemma 5.1 \cite{DJ}. Applying the method in \cite{DJJJY}, let $\gamma = \sqrt{6}(x_1+x_2-x_3)$ and $\epsilon=2(x_2+x_3)$, we have this identification:
\begin{proposition}
  As $V_{L_2}^{S_4}$- modules:
    \begin{equation}
        V_{\mathbb{Z}\gamma \cong ((V_{\mathbb{Z}\beta}^+)^0)^+ \oplus ((V_{\mathbb{Z}\beta}^+)^0)^-\oplus (V_{\mathbb{Z}\beta}^-)^+\oplus (V_{\mathbb{Z}\beta}^-)^-} \label{6.21}
    \end{equation}
    \begin{equation}
        V_{\mathbb{Z}\gamma+\frac{\gamma}{3}}\cong (V_{\mathbb{Z}\beta}^+)^1\oplus (V_{\mathbb{Z}\beta}^-)^+\oplus (V_{\mathbb{Z}\beta}^-)^- \label{6.22}
    \end{equation}
    \begin{equation}
         V_{\mathbb{Z}\gamma+\frac{2\gamma}{3}}\cong (V_{\mathbb{Z}\beta}^+)^1\oplus (V_{\mathbb{Z}\beta}^-)^+\oplus (V_{\mathbb{Z}\beta}^-)^- \label{6.23}
    \end{equation}
    \begin{equation}
        V_{\mathbb{Z}\gamma+\frac{\gamma}{6}}\cong (V_{\mathbb{Z}\beta+\frac{\beta}{4}}^0)^+ \oplus (V_{\mathbb{Z}\beta+\frac{\beta}{4}}^0)^- \oplus (V_{\mathbb{Z}\beta+\frac{\beta}{4}}^1) \label{6.24}
    \end{equation}
    \begin{equation}
        V_{\mathbb{Z}\gamma+\frac{\gamma}{2}}\cong 2V_{\mathbb{Z}\beta+\frac{\beta}{4}}^1 \label{6.25}
    \end{equation}
    \begin{equation}
        V_{\mathbb{Z}\gamma+\frac{5\gamma}{6}}\cong (V_{\mathbb{Z}\beta+\frac{\beta}{4}}^0)^+ \oplus (V_{\mathbb{Z}\beta+\frac{\beta}{4}}^0)^- \oplus (V_{\mathbb{Z}\beta+\frac{\beta}{4}}^1) \label{6.26}
    \end{equation}
    \begin{equation}
        V_{\mathbb{Z}\epsilon}\cong ((V_{\mathbb{Z}\beta}^+)^0)^+ \oplus (V_{\mathbb{Z}\beta}^+)^1 \oplus (V_{\mathbb{Z}\beta}^-)^+\oplus 2(V_{\mathbb{Z}\beta}^-)^- \label{6.27}
    \end{equation}
    \begin{equation}
       V_{\mathbb{Z}\epsilon+\frac{\epsilon}{4}}\cong  (V_{\mathbb{Z}\beta+\frac{\beta}{4}}^0)^+ \oplus (V_{\mathbb{Z}\beta+\frac{\beta}{4}}^0)^- \oplus 2(V_{\mathbb{Z}\beta+\frac{\beta}{4}}^1) \label{6.28}
    \end{equation}
    \begin{equation}
        V_{\mathbb{Z}\epsilon+\frac{\epsilon}{2}}\cong ((V_{\mathbb{Z}\beta}^+)^0)^- \oplus (V_{\mathbb{Z}\beta}^+)^1 \oplus 2(V_{\mathbb{Z}\beta}^-)^+\oplus (V_{\mathbb{Z}\beta}^-)^- \label{6.29}
    \end{equation}
    \begin{equation}
        V_{\mathbb{Z}\epsilon+\frac{3\epsilon}{4}}\cong  (V_{\mathbb{Z}\beta+\frac{\beta}{4}}^0)^+ \oplus (V_{\mathbb{Z}\beta+\frac{\beta}{4}}^0)^- \oplus 2(V_{\mathbb{Z}\beta+\frac{\beta}{4}}^1) \label{6.30}
    \end{equation}
\end{proposition}
\begin{proof}
    Applying the method in \cite{DJJJY} and  Corollary~6.6, we have
\begin{equation}
        V_{L_2}\cong ((V_{\mathbb{Z}\beta}^+)^0)^+\otimes W_1^0 
        \oplus ((V_{\mathbb{Z}\beta}^+)^0)^-\otimes W_1^1
        \oplus (V_{\mathbb{Z}\beta}^+)^1\otimes W_2^0
        \oplus (V_{\mathbb{Z}\beta}^-)^+\otimes W_3^1
        \oplus (V_{\mathbb{Z}\beta}^-)^-\otimes W_3^0,
        \label{6.31}
\end{equation}
and
\begin{equation}
        V_{L_2+\frac{\alpha}{2}}\cong (V_{\mathbb{Z}\beta+\frac{\beta}{4}}^0)^-\otimes W_2^1
        \oplus (V_{\mathbb{Z}\beta+\frac{\beta}{4}}^0)^+\otimes W_2^2
        \oplus (V_{\mathbb{Z}\beta+\frac{\beta}{4}}^1)\otimes W_4.
        \label{6.32}
\end{equation}

Here $W_i^j$ denotes an irreducible representation of $GL(2,3)$, where $i$ indicates the dimension of the irreducible module and $j$ distinguishes inequivalent irreducible modules of the same dimension. Combining Corollary~6.6 with the eigenspace decomposition (with respect to the eigenvalues of the linear map $P$) in \cite{DJJJY}, we obtain that $W_1^0$ has eigenvalue $1$, $W_1^1$ has eigenvalue $-1$, $W_2^0$ has eigenvalues $\pm 1$, $W_2^1$ has eigenvalues $e^{\pm\frac{3\pi}{4}i}$, $W_2^2$ has eigenvalues $e^{\pm\frac{\pi}{4}i}$, $W_3^0$ has eigenvalues $-1,\pm i$, $W_3^1$ has eigenvalues $1,\pm i$, and $W_4$ has eigenvalues $e^{\pm\frac{\pi}{4}i},e^{\pm\frac{3\pi}{4}i}$. 

Therefore, calculating the eigenvalues of $W_i^j$ with respect to $\rho$ and $\delta$ yields~(\ref{6.21})--(\ref{6.30}).

\end{proof}
\begin{lemma}
    \begin{equation}
        V_{\mathbb{Z}\zeta}^{T_2,+}\cong V_{\mathbb{Z}\epsilon+\frac{\epsilon}{8}} \label{6.33}
    \end{equation}
    \begin{equation}
        V_{\mathbb{Z}\zeta}^{T_2,-}\cong V_{\mathbb{Z}\epsilon+\frac{3\epsilon}{8}} \label{6.34}
    \end{equation}
\end{lemma}
\begin{proof}
    By Theorems~3.1--3.3, we know that any irreducible $V_{L_2}^{S_4}$-module arises from a decomposition of irreducible $V_{L_2}$-modules, $\tau_1$-twisted $V_{L_2}$-modules, $\delta$-twisted $V_{L_2}$-modules, $\rho$-twisted $V_{L_2}$-modules, or $P$-twisted $V_{L_2}$-modules. By \cite{WZ}, the irreducible $V_{L_2}^{S_4}$-modules arising from $\rho$-twisted and $P$-twisted $V_{L_2}$-modules are $M_{18}$--$M_{27}$ (using the indexing in Theorem~6.4).

Both $V_{L_2}$ and $V_{\mathbb{Z}\alpha+\frac{\alpha}{2}}$ are inequivalent irreducible $V_{L_2}$-modules, and both are stabilized by $P$. By \cite{DLM2}, there are exactly two $P$-twisted $V_{L_2}$-modules with conformal weights $\frac{1}{64}$ and $\frac{9}{64}$. Thus, the stabilizer of the action defined in Section~3 is equal to the centralizer of $P$, namely
\[
C_G(P)=\langle P\rangle \cong \mathbb{Z}_4.
\]
The Schur multiplier of $\mathbb{Z}_4$ is trivial, so there are eight inequivalent irreducible $V_{L_2}^{S_4}$-modules obtained from decomposing those two $P$-twisted $V$-modules.

Applying Theorems~3.1 and~3.2 to $G=\langle P\rangle$, we see that these eight modules are also inequivalent irreducible $V_{L_2}^{\langle P\rangle}=V_{\mathbb{Z}\zeta}$-modules. These eight modules correspond exactly to $M_{18}$--$M_{25}$. Consequently, $M_{26}$ and $M_{27}$ arise from $\rho$-twisted $V_{L_2}$-modules. By applying a similar argument, we see that $M_{26}$ and $M_{27}$ are also inequivalent irreducible $V_{L_2}^{\langle \rho\rangle}\cong V_{\mathbb{Z}\epsilon}$-modules. Comparing conformal weights, we obtain~(\ref{6.33}) and~(\ref{6.34}).
 
\end{proof}

The identifications above are useful for calculating fusion rules for modules that are contained in untwisted irreducible $V_{L_2}$-modules. In order to determine fusion rules involving irreducible $V_{L_2}^{S_4}$-modules that arise from twisted $V_{L_2}$-modules, we need to compute the $S$-matrix of $V_{L_2}^{S_4}$. The main method follows Lemma~5.1 of \cite{DJJJY}; for completeness, we provide a proof here.

\begin{lemma}
    The entries of the S-matrix that involves irreducible twisted modules of $V_{L_2}^{S_4}$ are given in the Appendix.
\end{lemma}
\begin{proof}
  We give a proof for $M^{18}-M^{25}$; the other cases are similar. 
Consider irreducible $V_{\mathbb{Z}\zeta}$-modules. They are of the form
\[
V_{\mathbb{Z}\zeta+\lambda_k}, \qquad \lambda_k=\frac{k\zeta}{32}, \quad k=0,\dots,31.
\]
By \cite[p.~106]{S}, we have
\begin{equation}
Z_{V_{\mathbb{Z}\zeta+\lambda_k}}\!\left(-\frac{1}{\tau}\right)
=\sum_{j=0}^{31}\frac{1}{\sqrt{32}}
e^{-2\pi i\langle \lambda_j,\lambda_k\rangle}
Z_{V_{\mathbb{Z}\zeta+\lambda_j}}(\tau).
\label{6.35}
\end{equation}
Thus,
\[
S_{\lambda_j,\lambda_k}
=\frac{1}{\sqrt{32}}e^{-2\pi i\langle \lambda_j,\lambda_k\rangle},
\]
where $S_{\lambda_j,\lambda_k}$ denotes the $(\lambda_j,\lambda_k)$-entry of the $S$-matrix of $V_{\mathbb{Z}\zeta}$.

Let $(S_{j,k})$ denote the $S$-matrix of $V_{L_2}^{S_4}$. For $j=18,\dots,25$, we have
\[
S_{j,0}=\frac{1}{\sqrt{32}}
=\mbox{qdim} M^j\cdot S_{0,0}
=6S_{0,0}.
\]
Therefore,
\[
S_{0,0}=\frac{1}{6\sqrt{32}},
\]
and hence
\[
S_{j,0}=\frac{\mbox{qdim} M^j}{6\sqrt{32}}
\qquad \text{for all } j=0,\dots,27.
\]

Fix $M^j\cong V_{\mathbb{Z}\zeta+\lambda_l}$ for some $l\in\{0,\dots,31\}$. 
Let $M^k$ be an irreducible $V_{L_2}^{S_4}$-module. 
If $M^k$ does not appear as a submodule of any $V_{\mathbb{Z}\zeta+\lambda_s}$, $s=0,\dots,31$, then
\[
S_{j,k}=0.
\]
Otherwise, suppose $M^k$ appears in
\[
V_{\mathbb{Z}\zeta+\lambda_{k_1}},\dots,V_{\mathbb{Z}\zeta+\lambda_{k_r}}
\]
with multiplicities $s_1,\dots,s_r$, respectively. Then
\[
S_{j,k}=\sum_{i=1}^{r}s_i\,S_{\lambda_l,\lambda_{k_i}}.
\]

The same method applies to compute the $S$-matrix entries involving
$M^{12}-M^{17}$ and $M^{26},M^{27}$.

\end{proof}
By proposition 6.4, $V_{\mathbb{Z}\beta+\frac{\beta}{8}}^{\pm}$ and $V_{\mathbb{Z}\beta+\frac{3\beta}{8}}^{\pm}$ can be also computed in this way.  The next lemma calculate the S matrix entries that relate to $(V_{\mathbb Z\beta}^{-})^{\pm}$.
\begin{lemma}
   \begin{eqnarray}
       Z_{(V_{\mathbb{Z}\beta}^-)^+}(-\frac{1}{\tau}) &= &\frac{1}{2\sqrt{32}}Z_{((V_{\mathbb{Z}\beta}^+)^0)^+ }(\tau) + \frac{1}{2\sqrt{32}}Z_{((V_{\mathbb{Z}\beta}^+)^0)^- }(\tau) + \frac{1}{\sqrt{32}}Z_{(V_{\mathbb{Z}\beta}^+)^1 }(\tau) + \frac{3}{2\sqrt{32}}Z_{(V_{\mathbb{Z}\beta}^-)^+ }(\tau) \nonumber \\
       & &+ \frac{3}{2\sqrt{32}}Z_{(V_{\mathbb{Z}\beta}^-)^- }(\tau) 
        + \frac{1}{\sqrt{32}}Z_{(V_{\mathbb{Z}\beta+\frac{\beta}{4}}^0)^+ }(\tau) + \frac{1}{\sqrt{32}}Z_{(V_{\mathbb{Z}\beta+\frac{\beta}{4}}^0)^- }(\tau)+\frac{2}{\sqrt{32}}Z_{(V_{\mathbb{Z}\beta+\frac{\beta}{4}}^1) }(\tau)\nonumber \\
        & &- \frac{1}{\sqrt{32}}Z_{(V_{\mathbb{Z}\beta+\frac{\beta}{8}})^+ }(\tau)  
        - \frac{1}{\sqrt{32}}Z_{(V_{\mathbb{Z}\beta+\frac{\beta}{8}})^- }(\tau) - \frac{1}{\sqrt{32}}Z_{(V_{\mathbb{Z}\beta+\frac{3\beta}{8}})^+ }(\tau) - \frac{1}{\sqrt{32}}Z_{(V_{\mathbb{Z}\beta+\frac{3\beta}{8}})^- }(\tau) \nonumber \\
       & & + \frac{1}{\sqrt{32}} \sum_{s=1,s\neq 0 \mbox{ mod }2}^{15} Z_{V_{\mathbb{Z}\zeta}+\frac{s\zeta}{32}}(\tau) - \frac{2}{\sqrt{32}}Z_{V_{{\mathbb{Z}\zeta}}^{T_2,+}}(\tau) - \frac{2}{\sqrt{32}}Z_{V_{{\mathbb{Z}\zeta}}^{T_2,-}}(\tau).\label{6.36}
   \end{eqnarray}
   \begin{eqnarray}
       Z_{(V_{\mathbb{Z}\beta}^-)^-}(-\frac{1}{\tau}) &= &\frac{1}{2\sqrt{32}}Z_{((V_{\mathbb{Z}\beta}^+)^0)^+ }(\tau) + \frac{1}{2\sqrt{32}}Z_{((V_{\mathbb{Z}\beta}^+)^0)^- }(\tau) + \frac{1}{\sqrt{32}}Z_{(V_{\mathbb{Z}\beta}^+)^1 }(\tau) + \frac{3}{2\sqrt{32}}Z_{(V_{\mathbb{Z}\beta}^-)^+ }(\tau) \nonumber \\
       & &+ \frac{3}{2\sqrt{32}}Z_{(V_{\mathbb{Z}\beta}^-)^- }(\tau) +\frac{1}{\sqrt{32}}Z_{(V_{\mathbb{Z}\beta+\frac{\beta}{4}}^0)^+ }(\tau) + \frac{1}{\sqrt{32}}Z_{(V_{\mathbb{Z}\beta+\frac{\beta}{4}}^0)^- }(\tau)+\frac{2}{\sqrt{32}}Z_{(V_{\mathbb{Z}\beta+\frac{\beta}{4}}^1) }(\tau)\nonumber\\
       & &
        - \frac{1}{\sqrt{32}}Z_{(V_{\mathbb{Z}\beta+\frac{\beta}{8}})^+ }(\tau)  
        - \frac{1}{\sqrt{32}}Z_{(V_{\mathbb{Z}\beta+\frac{\beta}{8}})^- }(\tau) - \frac{1}{\sqrt{32}}Z_{(V_{\mathbb{Z}\beta+\frac{3\beta}{8}})^+ }(\tau)\nonumber \\
        & &- \frac{1}{\sqrt{32}}Z_{(V_{\mathbb{Z}\beta+\frac{3\beta}{8}})^- }(\tau) 
        - \frac{1}{\sqrt{32}} \sum_{s=1,s\neq 0 \mbox{ mod }2}^{15} Z_{V_{\mathbb{Z}\zeta}+\frac{s\zeta}{32}}(\tau) + \frac{2}{\sqrt{32}}Z_{V_{{\mathbb{Z}\zeta}}^{T_2,+}}(\tau)\nonumber \\
        & &+ \frac{2}{\sqrt{32}}Z_{V_{{\mathbb{Z}\zeta}}^{T_2,-}}(\tau).
   \end{eqnarray}\label{6.37}
\end{lemma}

\begin{proof}
    Let us start with proof of (\ref{6.36}). By theorem 6.4, we have $(V_{\mathbb{Z}\beta}^-)^+\cong V_{\mathbb{Z}\zeta}^-$ as $V_{L_2}^{S_4}$ module. So, 
    \begin{eqnarray}
        Z_{V_{\mathbb{Z}\zeta}}(-\frac{1}{\tau})= Z_{V_{\mathbb{Z}\zeta}^+} (-\frac{1}{\tau})+ Z_{V_{\mathbb{Z}\zeta}^-}(-\frac{1}{\tau})=\sum_{j=0}^{31}\frac{1}{\sqrt{32}} Z_{V_{\mathbb{Z}\zeta+\frac{j}{32}\zeta}} (\tau).
    \end{eqnarray}
    By theorem 5.6, theorem 5.7, we can calculate quantum dimension of $V_{\mathbb{Z}\zeta}^+$ modules. $V_{\mathbb{Z}\zeta}^{\pm},V_{\mathbb{\zeta}+\frac{\zeta}{2}}^{\pm}$ have quantum dimensions 1, $V_{\mathbb{Z}\zeta+\frac{s\zeta}{32}}$ have quantum dimensions 2 for $(s=1,...,15)$ and $V_{\mathbb{Z}\zeta}^{T_1,\pm},V_{\mathbb{Z}\zeta}^{T_2,\pm}$ have quantum  dimensions 4. Apply technique in previous lemma, we know $S_{0,0}=\frac{1}{2\sqrt{32}}$. Thus, we have 
    \begin{eqnarray}
        Z_{V_{\mathbb{Z}\zeta}^+}(-\frac{1}{\tau})&=&\frac{1}{2\sqrt{32}}Z_{V_{\mathbb{Z}\zeta}^+}(\tau)+\frac{1}{2\sqrt{32}}Z_{V_{\mathbb{Z}\zeta}^-}(\tau)+\frac{1}{2\sqrt{32}}Z_{V_{\mathbb{Z}\zeta+\frac{\zeta}{2}}^+}(\tau)+\frac{1}{2\sqrt{32}}Z_{V_{\mathbb{Z}\zeta+\frac{\zeta}{2}}^-}(\tau)\nonumber\\
        & &+ \sum_{j=1}^{15}\frac{1}{\sqrt{32}} Z_{V_{\mathbb{Z}\zeta+\frac{j}{32}\zeta}} (\tau)+\frac{2}{\sqrt{32}}Z_{V_{\mathbb{Z}\zeta}^{T_1,+}}(\tau)+\frac{2}{\sqrt{32}}Z_{V_{\mathbb{Z}\zeta}^{T_1,-}}(\tau)+\frac{2}{\sqrt{32}}Z_{V_{\mathbb{Z}\zeta}^{T_2,+}}(\tau)\nonumber \\
        & &+\frac{2}{\sqrt{32}}Z_{V_{\mathbb{Z}\zeta}^{T_2,-}}(\tau).
    \end{eqnarray}
    Then, we have 
    \begin{eqnarray}
        Z_{V_{\mathbb{Z}\zeta}^-}(-\frac{1}{\tau})&=&\frac{1}{2\sqrt{32}}Z_{V_{\mathbb{Z}\zeta}^+}(\tau)+\frac{1}{2\sqrt{32}}Z_{V_{\mathbb{Z}\zeta}^-}(\tau)+\frac{1}{2\sqrt{32}}Z_{V_{\mathbb{Z}\zeta+\frac{\zeta}{2}}^+}(\tau)+\frac{1}{2\sqrt{32}}Z_{V_{\mathbb{Z}\zeta+\frac{\zeta}{2}}^-}(\tau)\nonumber\\
        & &+ \sum_{j=1}^{15}\frac{1}{\sqrt{32}} Z_{V_{\mathbb{Z}\zeta+\frac{j}{32}\zeta}} (\tau)-\frac{2}{\sqrt{32}}Z_{V_{\mathbb{Z}\zeta}^{T_1,+}}(\tau)-\frac{2}{\sqrt{32}}Z_{V_{\mathbb{Z}\zeta}^{T_1,-}}(\tau)-\frac{2}{\sqrt{32}}Z_{V_{\mathbb{Z}\zeta}^{T_2,+}}(\tau)\nonumber \\
        & &-\frac{2}{\sqrt{32}}Z_{V_{\mathbb{Z}\zeta}^{T_2,-}}(\tau).
    \end{eqnarray}
    Combine with proposition 6.4, we get (\ref{6.36}).Prove (6.44) is similar. We can compute the quantum dimensions of $V_{\mathbb{Z}\beta}^+$ modules. $V_{\mathbb{Z}\beta}^{\pm},V_{\mathbb{Z}\beta+\frac{\beta}{2}}^{\pm}$ have quantum dimensions 1, and $V_{\mathbb{Z}\beta+\frac{s\beta}{8}},V_{\mathbb{Z\beta}}^{T_1,\pm},V_{\mathbb{Z}\beta}^{T_2,\pm}$ have quantum dimensions 2. We have 
    \begin{equation}
        Z_{V_{\mathbb{Z}\beta}}(-\frac{1}{\tau})=Z_{V_{\mathbb{Z}\beta}^+}(-\frac{1}{\tau})+Z_{V_{\mathbb{Z}\beta}^-}(-\frac{1}{\tau})=\sum_{i=0}^{7}\frac{1}{\sqrt{8}}Z_{V_{\mathbb{Z}\beta+\frac{i\beta}{8}}}(\tau)
    \end{equation}
     and \begin{eqnarray}
         Z_{V_{\mathbb{Z}\beta}^+}(-\frac{1}{\tau})&=& \frac{1}{2\sqrt{8}}(Z_{V_{\mathbb{Z}\beta}^+}(\tau)+Z_{V_{\mathbb{Z}\beta}^-}(\tau)+Z_{V_{\mathbb{Z}\beta+\frac{\beta}{2}}^+}(\tau)+Z_{V_{\mathbb{Z}\beta+\frac{\beta}{2}}^-}(\tau))+\frac{1}{\sqrt{8}}\sum_{i=1}^{3}Z_{V_{\mathbb{Z}\beta}+\frac{i}{8}\beta}(\tau)\nonumber\\
         & &+\frac{1}{\sqrt{8}}Z_{V_{\mathbb{Z}\beta}^{T_1,+}}(\tau)+\frac{1}{\sqrt{8}}Z_{V_{\mathbb{Z}\beta}^{T_1,-}}(\tau)+\frac{1}{\sqrt{8}}Z_{V_{\mathbb{Z}\beta}^{T_2,+}}(\tau)+\frac{1}{\sqrt{8}}Z_{V_{\mathbb{Z}\beta}^{T_2,-}}(\tau).
     \end{eqnarray}
     Thus, \begin{eqnarray}
         Z_{V_{\mathbb{Z}\beta}^-}(-\frac{1}{\tau})&=& \frac{1}{2\sqrt{8}}(Z_{V_{\mathbb{Z}\beta}^+}(\tau)+Z_{V_{\mathbb{Z}\beta}^-}(\tau)+Z_{V_{\mathbb{Z}\beta+\frac{\beta}{2}}^+}(\tau)+Z_{V_{\mathbb{Z}\beta+\frac{\beta}{2}}^-}(\tau))+\frac{1}{\sqrt{8}}\sum_{i=1}^{3}Z_{V_{\mathbb{Z}\beta}+\frac{i}{8}\beta}(\tau)\nonumber\\
         & &-\frac{1}{\sqrt{8}}Z_{V_{\mathbb{Z}\beta}^{T_1,+}}(\tau)-\frac{1}{\sqrt{8}}Z_{V_{\mathbb{Z}\beta}^{T_1,-}}(\tau)-\frac{1}{\sqrt{8}}Z_{V_{\mathbb{Z}\beta}^{T_2,+}}(\tau)-\frac{1}{\sqrt{8}}Z_{V_{\mathbb{Z}\beta}^{T_2,-}}(\tau).
     \end{eqnarray}
     By Proposition 6.4, we have 
     \begin{eqnarray}
         Z_{V_{\mathbb{Z}\beta}^-}(-\frac{1}{\tau})&=&\frac{1}{2\sqrt{8}}(Z_{{((V_{\mathbb{Z}\beta}^+})^0)^+}(\tau)+Z_{{((V_{\mathbb{Z}\beta}^+})^0)^-}(\tau)+2Z_{{(V_{\mathbb{Z}\beta}^+})^1}(\tau))+\frac{3}{2\sqrt{8}}(Z_{(V_{\mathbb{Z}\beta}^-)^+}(\tau)+Z_{(V_{\mathbb{Z}\beta}^-)^-}(\tau))\nonumber\\
         & &+\frac{1}{\sqrt{8}}(Z_{(V_{\mathbb{Z}\beta+\frac{\beta}{4}}^0)^+ }(\tau) +Z_{(V_{\mathbb{Z}\beta+\frac{\beta}{4}}^0)^- }(\tau) +2Z_{V_{\mathbb{Z}\beta+\frac{\beta}{4}}^1 }(\tau) )\nonumber\\
         & &-\frac{1}{\sqrt{8}}(Z_{(V_{\mathbb{Z}\beta}+\frac{\beta}{8})^+}(\tau)+Z_{(V_{\mathbb{Z}\beta}+\frac{\beta}{8})^-}(\tau)+Z_{(V_{\mathbb{Z}\beta}+\frac{3\beta}{8})^+}(\tau)+Z_{(V_{\mathbb{Z}\beta}+\frac{3\beta}{8})^-}(\tau))
     \end{eqnarray}
     Subtract (\ref{6.36}), we can get (6.38).
\end{proof}
\section{fusion rules for $V_{L_2}^{S_4}$}
\setcounter{equation}{0}
\numberwithin{equation} {section}
  In this section, we compute the full fusion rules for $V_{L_2}^{S_4}$-modules. The index of each module follow from Theorem~6.4. Throughout this section, we use $0$ to denote $+$ and $1$ to denote $-$ when appropriate. For example, $(V_{\mathbb{Z}\beta}^-)^0$ denotes $(V_{\mathbb{Z}\beta}^-)^+$, and $(V_{\mathbb{Z}\beta}^-)^1$ denotes $(V_{\mathbb{Z}\beta}^-)^-$.

\begin{theorem}
    \begin{equation}
        (1). (V_{\mathbb{Z}\beta+\frac{i\beta}{8}}^j)\boxtimes(V_{\mathbb{Z}\gamma+\frac{r}{18}\gamma})=\bigoplus _{r}V_{\mathbb{Z}\gamma+\frac{r\gamma
        }{18}}.(i=1,3,j=0,1,1\leq r\leq 8,r\neq 0 \mbox{ mod }3.)\label{7.1}
    \end{equation}
    \begin{eqnarray}
     (2). & &  M_8\boxtimes M_{18}=M_{18}\oplus M_{19}\oplus V_{\mathbb{Z}\zeta}^{T_2,+}\oplus V_{\mathbb{Z}\zeta}^{T_2,-},
        M_8\boxtimes M_{19}=M_{18}\oplus M_{20}\oplus V_{\mathbb{Z}\zeta}^{T_2,+}\oplus V_{\mathbb{Z}\zeta}^{T_2,-},\nonumber\\
       & & M_8\boxtimes M_{20}=M_{19}\oplus M_{21}\oplus V_{\mathbb{Z}\zeta}^{T_2,+}\oplus V_{\mathbb{Z}\zeta}^{T_2,-},
       M_8\boxtimes M_{21}=M_{20}\oplus M_{22}\oplus V_{\mathbb{Z}\zeta}^{T_2,+}\oplus V_{\mathbb{Z}\zeta}^{T_2,-}\nonumber\\
       & & M_8\boxtimes M_{22}=M_{21}\oplus M_{23}\oplus V_{\mathbb{Z}\zeta}^{T_2,+}\oplus V_{\mathbb{Z}\zeta}^{T_2,-},
       M_8\boxtimes M_{23}=M_{22}\oplus M_{24}\oplus V_{\mathbb{Z}\zeta}^{T_2,+}\oplus V_{\mathbb{Z}\zeta}^{T_2,-}\nonumber\\
       & &M_8\boxtimes M_{24}=M_{23}\oplus M_{25}\oplus V_{\mathbb{Z}\zeta}^{T_2,+}\oplus V_{\mathbb{Z}\zeta}^{T_2,-},
       M_8\boxtimes M_{25}=M_{24}\oplus M_{25}\oplus V_{\mathbb{Z}\zeta}^{T_2,+}\oplus V_{\mathbb{Z}\zeta}^{T_2,-}\nonumber\\       %
       & & M_9\boxtimes M_{18}=M_{24}\oplus M_{25}\oplus V_{\mathbb{Z}\zeta}^{T_2,+}\oplus V_{\mathbb{Z}\zeta}^{T_2,-},
       M_9\boxtimes M_{19}=M_{23}\oplus M_{25}\oplus V_{\mathbb{Z}\zeta}^{T_2,+}\oplus V_{\mathbb{Z}\zeta}^{T_2,-}\nonumber\\
       & & M_9\boxtimes M_{20}=M_{22}\oplus M_{24}\oplus V_{\mathbb{Z}\zeta}^{T_2,+}\oplus V_{\mathbb{Z}\zeta}^{T_2,-},
       M_9\boxtimes M_{21}=M_{21}\oplus M_{23}\oplus V_{\mathbb{Z}\zeta}^{T_2,+}\oplus V_{\mathbb{Z}\zeta}^{T_2,-}\nonumber\\
       & &M_9\boxtimes M_{22}=M_{20}\oplus M_{22}\oplus V_{\mathbb{Z}\zeta}^{T_2,+}\oplus V_{\mathbb{Z}\zeta}^{T_2,-},
       M_9\boxtimes M_{23}=M_{19}\oplus M_{21}\oplus V_{\mathbb{Z}\zeta}^{T_2,+}\oplus V_{\mathbb{Z}\zeta}^{T_2,-}\nonumber\\
       & &M_9\boxtimes M_{24}=M_{18}\oplus M_{20}\oplus V_{\mathbb{Z}\zeta}^{T_2,+}\oplus V_{\mathbb{Z}\zeta}^{T_2,-},
       M_9\boxtimes M_{25}=M_{18}\oplus M_{19}\oplus V_{\mathbb{Z}\zeta}^{T_2,+}\oplus V_{\mathbb{Z}\zeta}^{T_2,-}\nonumber\\     %
       & & M_{10}\boxtimes M_{18}=M_{20}\oplus M_{21}\oplus V_{\mathbb{Z}\zeta}^{T_2,+}\oplus V_{\mathbb{Z}\zeta}^{T_2,-},
       M_{10}\boxtimes M_{19}=M_{19}\oplus M_{22}\oplus V_{\mathbb{Z}\zeta}^{T_2,+}\oplus V_{\mathbb{Z}\zeta}^{T_2,-}\nonumber\\
       & &M_{10}\boxtimes M_{20}=M_{18}\oplus M_{23}\oplus V_{\mathbb{Z}\zeta}^{T_2,+}\oplus V_{\mathbb{Z}\zeta}^{T_2,-},
       M_{10}\boxtimes M_{21}=M_{18}\oplus M_{24}\oplus V_{\mathbb{Z}\zeta}^{T_2,+}\oplus V_{\mathbb{Z}\zeta}^{T_2,-}\nonumber\\
       & & M_{10}\boxtimes M_{22}=M_{19}\oplus M_{25}\oplus V_{\mathbb{Z}\zeta}^{T_2,+}\oplus V_{\mathbb{Z}\zeta}^{T_2,-},
       M_{10}\boxtimes M_{23}=M_{20}\oplus M_{25}\oplus V_{\mathbb{Z}\zeta}^{T_2,+}\oplus V_{\mathbb{Z}\zeta}^{T_2,-}\nonumber\\
       & & M_{10}\boxtimes M_{24}=M_{21}\oplus M_{24}\oplus V_{\mathbb{Z}\zeta}^{T_2,+}\oplus V_{\mathbb{Z}\zeta}^{T_2,-},
       M_{10}\boxtimes M_{25}=M_{22}\oplus M_{23}\oplus V_{\mathbb{Z}\zeta}^{T_2,+}\oplus V_{\mathbb{Z}\zeta}^{T_2,-}\nonumber\\    %
       & & M_{11}\boxtimes M_{18}=M_{22}\oplus M_{23}\oplus V_{\mathbb{Z}\zeta}^{T_2,+}\oplus V_{\mathbb{Z}\zeta}^{T_2,-},
       M_{11}\boxtimes M_{19}=M_{21}\oplus M_{24}\oplus V_{\mathbb{Z}\zeta}^{T_2,+}\oplus V_{\mathbb{Z}\zeta}^{T_2,-}\nonumber\\
       & &M_{11}\boxtimes M_{20}=M_{20}\oplus M_{25}\oplus V_{\mathbb{Z}\zeta}^{T_2,+}\oplus V_{\mathbb{Z}\zeta}^{T_2,-},
       M_{11}\boxtimes M_{21}=M_{19}\oplus M_{25}\oplus V_{\mathbb{Z}\zeta}^{T_2,+}\oplus V_{\mathbb{Z}\zeta}^{T_2,-}\nonumber\\
       & & M_{11}\boxtimes M_{22}=M_{18}\oplus M_{24}\oplus V_{\mathbb{Z}\zeta}^{T_2,+}\oplus V_{\mathbb{Z}\zeta}^{T_2,-},
       M_{11}\boxtimes M_{23}=M_{18}\oplus M_{23}\oplus V_{\mathbb{Z}\zeta}^{T_2,+}\oplus V_{\mathbb{Z}\zeta}^{T_2,-}\nonumber\\
       & &M_{11}\boxtimes M_{24}=M_{19}\oplus M_{22}\oplus V_{\mathbb{Z}\zeta}^{T_2,+}\oplus V_{\mathbb{Z}\zeta}^{T_2,-},
       M_{11}\boxtimes M_{25}=M_{20}\oplus M_{21}\oplus V_{\mathbb{Z}\zeta}^{T_2,+}\oplus V_{\mathbb{Z}\zeta}^{T_2,-}\nonumber\\
       & & \label{7.2}
    \end{eqnarray}
    \begin{equation}
        (3) .(V_{\mathbb{Z}\beta+\frac{i\beta}{8}}^j)\boxtimes V_{\mathbb{Z}\zeta }^{T_2,k}=(\bigoplus_{s=1,s\neq 0\mbox{ mod }2}^{15} V_{\frac{s\zeta}{32}+\mathbb{Z}\zeta})\oplus V_{\mathbb{Z}\zeta}^{T_2,+} \oplus V_{\mathbb{Z}\zeta}^{T_2,-}.(i=1,3, j=0,1, k=0,1) \label{7.3}
    \end{equation}
    (4).\quad For the fusion rules of type $M_i\boxtimes M_j$ with $12\leq i \leq 17,18\leq j \leq 25$.
    \begin{align}
& M_{12}\boxtimes M_{18}=M_{12}\boxtimes M_{21}=M_{12}\boxtimes M_{22}
 =M_{12}\boxtimes M_{25}=M_{13}\boxtimes M_{19} \nonumber\\
& =M_{13}\boxtimes M_{20}=M_{13}\boxtimes M_{23}
 =M_{13}\boxtimes M_{24}=M_{14}\boxtimes M_{19} \nonumber\\
& =M_{14}\boxtimes M_{20}=M_{14}\boxtimes M_{23}
 =M_{14}\boxtimes M_{24}=M_{15}\boxtimes M_{18} \nonumber\\
& =M_{15}\boxtimes M_{21}
 =M_{15}\boxtimes M_{22}= M_{15}\boxtimes M_{25}
 =M_{16}\boxtimes M_{18} \nonumber\\
& =M_{16}\boxtimes M_{21}=M_{16}\boxtimes M_{22}
 =M_{16}\boxtimes M_{25}=M_{17}\boxtimes M_{19} \nonumber\\
& =M_{17}\boxtimes M_{21}=M_{17}\boxtimes M_{22}
 =M_{17}\boxtimes M_{25} \nonumber\\
& =M_{18}\oplus M_{21}\oplus M_{22}\oplus M_{25}
 \oplus M_{26}\oplus M_{27}.
\label{7.4}
\end{align}
    All other products of this type are equal to: $M_{19}\oplus M_{20}\oplus M_{23}\oplus M_{24}\oplus M_{26}\oplus M_{27}$. 
    \begin{equation}\label{7.5}
\begin{aligned}
(5)\quad
V_{\mathbb{Z}\gamma+\frac{r\gamma}{18}}
\boxtimes V_{\mathbb{Z}\zeta}^{T_2,j}
&=
\bigoplus_{\substack{1\leq s\leq 15\\
                     s\not\equiv 0\pmod{2}}}
V_{\mathbb{Z}\zeta+\frac{s\zeta}{32}}
\\
&\quad\oplus 2\,V_{\mathbb{Z}\zeta}^{T_2,+}
      \oplus 2\,V_{\mathbb{Z}\zeta}^{T_2,-},
\\[-2pt]
&\qquad
\left(
1\leq r\leq 8,\quad
r\not\equiv 0\pmod{3},\quad
j=0,1
\right).
\end{aligned}
\end{equation}
    \begin{equation}\label{7.6}
\begin{aligned}
(6)\quad
V_{\mathbb{Z}\zeta+\frac{s\zeta}{32}}
\boxtimes V_{\mathbb{Z}\zeta}^{T_2,j}
&=
\bigoplus_{\substack{1\leq r\leq 8\\
                     r\not\equiv 0\pmod{3}}}
V_{\mathbb{Z}\gamma+\frac{r\gamma}{18}}
\\
&\quad\oplus
\bigoplus_{\substack{a=1,3\\ t=0,1}}
V_{\frac{a\beta}{8}+\mathbb{Z}\beta}^{t},
\\[-2pt]
&\qquad
\left(
1\leq s\leq 15,\quad
s\not\equiv 0\pmod{2}
\right).
\end{aligned}
\end{equation}
    \begin{eqnarray}
        (7).(V_{\mathbb{Z}\beta}^{-})^i\boxtimes V_{\mathbb{Z}\zeta}^{T_2,j}=V_{\mathbb{Z}\zeta}^{T_2,+}\oplus 2V_{\mathbb{Z}\zeta}^{T_2,-} \mbox{ for } i=j,(i,j=0,1). \label{7.7}
    \end{eqnarray}
    \begin{eqnarray}
        (8).(V_{\mathbb{Z}\beta}^{-})^i\boxtimes V_{\mathbb{Z}\zeta}^{T_2,j}=2V_{\mathbb{Z}\zeta}^{T_2,+}\oplus V_{\mathbb{Z}\zeta}^{T_2,-} \mbox{ for } i\neq j,(i,j=0,1).\label{7.8}
    \end{eqnarray}
    
\end{theorem}
  \begin{proof}
      This directly follows from theorem 4.3, lemma 6.9, lemma 6.10.
  \end{proof}
  Here are some conventions we need to take for the next theorem. Let $i,j\in \mathbb{Z}$, then $\overline {i+j}$ denotes $i+j$ mod 2 and $\overline{\overline{i+j}}$ denotes $i+j$ mod 3. By \cite{DJJJY} and \cite{WZ}, we have following identification as $V_{L_2}^{S_4}$ modules:
  \begin{equation}
      W_{\delta,1}^{0}\cong V_{\mathbb{Z}\gamma+\frac{\gamma}{18}}\label{7.9}
  \end{equation}
  \begin{equation}
      W_{\delta,2}^{0}\cong V_{\mathbb{Z}\gamma+\frac{2\gamma}{18}}\label{7.10}
  \end{equation}
  \begin{equation}
      W_{\delta,2}^{1}\cong V_{\mathbb{Z}\gamma+\frac{4\gamma}{18}}\label{7.11}
  \end{equation}
  \begin{equation}
      W_{\delta,1}^{1}\cong V_{\mathbb{Z}\gamma+\frac{5\gamma}{18}}\label{7.12}
  \end{equation}
  \begin{equation}
      W_{\delta,1}^{2}\cong V_{\mathbb{Z}\gamma+\frac{7\gamma}{18}}\label{7.13}
  \end{equation}
  \begin{equation}
      W_{\delta,2}^{2}\cong V_{\mathbb{Z}\gamma+\frac{8\gamma}{18}}\label{7.14}
  \end{equation}
  \begin{theorem}
      \begin{equation}
          (1).((V_{\mathbb{Z}\beta}^+)^0)^i\boxtimes ((V_{\mathbb{Z}\beta}^+)^0)^j = ((V_{\mathbb{Z}\beta}^+)^0)^{\overline{i+j}}.(i,j=0,1)\label {7.15}
      \end{equation}
      \begin{equation}
          (2).((V_{\mathbb{Z}\beta}^+)^0)^i\boxtimes (V_{\mathbb{Z}\beta}^+)^1 = (V_{\mathbb{Z}\beta}^+)^1. (i=0,1)\label{7.16}
      \end{equation}
      \begin{equation}
          (3).((V_{\mathbb{Z}\beta}^+)^0)^i\boxtimes (V_{\mathbb{Z}\beta}^-)^j=(V_{\mathbb{Z}\beta}^-)^{\overline{i+j}}.\label{7.17}(i,j=0,1)
      \end{equation}
      \begin{equation}
          (4).((V_{\mathbb{Z}\beta}^+)^0)^i\boxtimes (V_{\mathbb{Z}\beta+\frac{\beta}{4}}^0)^j=(V_{\mathbb{Z}\beta+\frac{\beta}{4}}^0)^{\overline{i+j}}.(i,j=0,1)\label{7.18}
      \end{equation}
      \begin{equation}
          (5).((V_{\mathbb{Z}\beta}^+)^0)^i \boxtimes V_{\mathbb{Z}\beta+\frac{\beta}{4}}^1 = V_{\mathbb{Z}\beta+\frac{\beta}{4}}^1. (i=0,1)\label{7.19}
      \end{equation}
      \begin{equation}
          (6).((V_{\mathbb{Z}\beta}^+)^0)^i\boxtimes (V_{\mathbb{Z}\beta+\frac{j}{8}\beta})^k=(V_{\mathbb{Z}\beta+\frac{j}{8}\beta})^{\overline{i+k}}.(i,k=0,1, j=1,3)\label{7.20}
      \end{equation}
      \begin{equation}
          (7).((V_{\mathbb{Z}\beta}^+)^0)^i\boxtimes V_{\mathbb{Z}\gamma+\frac{r\gamma}{18}}=V_{\mathbb{Z}\gamma+\frac{r\gamma}{18}}.(i=0,1)\label{7.21}
      \end{equation}
      \begin{equation}
          (8).((V_{\mathbb{Z}\beta}^+)^0)^i\boxtimes V_{\mathbb{Z}\zeta+\frac{s\zeta}{32}}= V_{\mathbb{Z}\zeta+\frac{s\zeta}{32}}.(i=0)\label{7.22}
      \end{equation}
      \begin{equation}
          (9).((V_{\mathbb{Z}\beta}^+)^0)^i\boxtimes V_{\mathbb{Z}\zeta+\frac{s\zeta}{32}}= V_{\mathbb{Z}\zeta+\frac{(16-s)\zeta}{32}}.(i=1)\label{7.23}
      \end{equation}
      \begin{eqnarray}
          (10).((V_{\mathbb{Z}\beta}^+)^0)^i\boxtimes V_{\mathbb{Z}\zeta}^{T_2,j}=V_{\mathbb{Z}\zeta}^{T_2,\overline{i+j}}.(i,j=0,1)\label{7.24}
      \end{eqnarray}

      \begin{equation}
          (11). (V_{\mathbb{Z}\beta}^+)^1\boxtimes (V_{\mathbb{Z}\beta}^+)^1=((V_{\mathbb{Z}\beta}^+)^0)^+\oplus ((V_{\mathbb{Z}\beta}^+)^0)^-\oplus (V_{\mathbb{Z}\beta}^+)^1. \label {7.25}
      \end{equation}
      \begin{equation}
          (12).(V_{\mathbb{Z}\beta}^+)^1 \boxtimes (V_{\mathbb{Z}\beta}^-)^j=(V_{\mathbb{Z}\beta}^-)^+\oplus (V_{\mathbb{Z}\beta}^-)^-.\label{7.26}(j=0,1)
      \end{equation}
      \begin{equation}
          (13).(V_{\mathbb{Z}\beta}^+)^1\boxtimes (V_{\mathbb{Z}\beta+\frac{\beta}{4}}^0)^i=V_{\mathbb{Z}\beta+\frac{\beta}{4}}^1.\label{7.27}(i=0,1)
      \end{equation}
      \begin{equation}
          (14).(V_{\mathbb{Z}\beta}^+)^1\boxtimes V_{\mathbb{Z}\beta+\frac{\beta}{4}}^1=(V_{\mathbb{Z}\beta+\frac{\beta}{4}}^0)^+\oplus (V_{\mathbb{Z}\beta+\frac{\beta}{4}}^0)^-\oplus V_{\mathbb{Z}\beta+\frac{\beta}{4}}^1.\label{7.27a}
      \end{equation}
      \begin{equation}
          (15).(V_{\mathbb{Z}\beta}^+)^1\boxtimes (V_{\mathbb{Z}\beta+\frac{j\beta}{8}})^k=(V_{\mathbb{Z}\beta+\frac{j\beta}{8}})^+\oplus (V_{\mathbb{Z}\beta+\frac{j\beta}{8}})^-.\label{7.28}(j=1,3,k=0,1)
      \end{equation}
      \begin{equation}
          (16).(V_{\mathbb{Z}\beta}^+)^1\boxtimes V_{\mathbb{Z}\gamma+\frac{\gamma}{18}}=V_{\mathbb{Z}\gamma+\frac{5\gamma}{18}}\oplus V_{\mathbb{Z}\gamma+\frac{7\gamma}{18}}.\label{7.29}
      \end{equation}
      \begin{equation}
          (V_{\mathbb{Z}\beta}^+)^1\boxtimes V_{\mathbb{Z}\gamma+\frac{2\gamma}{18}}=V_{\mathbb{Z}\gamma+\frac{4\gamma}{18}}\oplus V_{\mathbb{Z}\gamma+\frac{8\gamma}{18}}.\label{7.30}
      \end{equation}
      \begin{equation}
          (V_{\mathbb{Z}\beta}^+)^1\boxtimes V_{\mathbb{Z}\gamma+\frac{4\gamma}{18}}=V_{\mathbb{Z}\gamma+\frac{2\gamma}{18}}\oplus V_{\mathbb{Z}\gamma+\frac{8\gamma}{18}}.\label{7.31}
      \end{equation}
      \begin{equation}
          (V_{\mathbb{Z}\beta}^+)^1\boxtimes V_{\mathbb{Z}\gamma+\frac{5\gamma}{18}}=V_{\mathbb{Z}\gamma+\frac{\gamma}{18}}\oplus V_{\mathbb{Z}\gamma+\frac{7\gamma}{18}}.\label{7.32}
      \end{equation}
      \begin{equation}
          (V_{\mathbb{Z}\beta}^+)^1\boxtimes V_{\mathbb{Z}\gamma+\frac{7\gamma}{18}}=V_{\mathbb{Z}\gamma+\frac{\gamma}{18}}\oplus V_{\mathbb{Z}\gamma+\frac{5\gamma}{18}}.\label{7.33}
      \end{equation}
      \begin{equation}
          (V_{\mathbb{Z}\beta}^+)^1\boxtimes V_{\mathbb{Z}\gamma+\frac{8\gamma}{18}}=V_{\mathbb{Z}\gamma+\frac{2\gamma}{18}}\oplus V_{\mathbb{Z}\gamma+\frac{4\gamma}{18}}.\label{7.34}
      \end{equation}
      \begin{equation}
          (17).(V_{\mathbb{Z}\beta}^+)^1\boxtimes V_{\mathbb{Z}\zeta +\frac{s\zeta}{32}}=V_{\frac{s\zeta}{32}+\mathbb{Z}\zeta}\oplus V_{\mathbb{Z}\zeta+\frac{(16-s)\zeta}{32}}.\label{7.35}
      \end{equation}
      \begin{equation}
          (18).(V_{\mathbb{Z}\beta}^+)^1\boxtimes V_{\mathbb{Z}\zeta}^{T_2,\pm}=V_{\mathbb{Z}\zeta}^{T_2,+}\oplus V_{\mathbb{Z}\zeta}^{T_2,-}.\label{7.36}
      \end{equation}

      \begin{equation}
          (19). (V_{\mathbb{Z}\beta}^-)^i\boxtimes (V_{\mathbb{Z}\beta}^-)^j= ((V_{\mathbb{Z}\beta}^+)^0)^{\overline{i+j}}\oplus (V_{\mathbb{Z}\beta}^+)^1\oplus (V_{\mathbb{Z}\beta}^-)^+\oplus (V_{\mathbb{Z}\beta}^-)^-.\label {7.37}(i,j=0,1)
      \end{equation}
      \begin{equation}
          (20).(V_{\mathbb{Z}\beta}^-)^i\boxtimes (V_{\mathbb{Z}\beta+\frac{\beta}{4}}^0)^j= (V_{\mathbb{Z}\beta+\frac{\beta}{4}}^0)^{\overline{i+j}}\oplus V_{\mathbb{Z}\beta+\frac{\beta}{4}}^1.\label{7.38}
      \end{equation}
      \begin{equation}
         (21). (V_{\mathbb{Z}\beta}^-)^i\boxtimes V_{\mathbb{Z}\beta+\frac{\beta}{4}}^1= (V_{\mathbb{Z}\beta+\frac{\beta}{4}}^0)^{+}\oplus (V_{\mathbb{Z}\beta+\frac{\beta}{4}}^0)^-\oplus 2V_{\mathbb{Z}\beta+\frac{\beta}{4}}^1.\label{7.38a}
      \end{equation}
      \begin{equation}
          (22).(V_{\mathbb{Z}\beta}^-)^i\boxtimes (V_{\mathbb{Z}\beta+\frac{k\beta}{8}})^j=(V_{\mathbb{Z}\beta+\frac{k\beta}{8}})^{\overline{i+j}}\oplus V_{\mathbb{Z}\beta+\frac{(4-k)\beta}{8}}^+\oplus V_{\mathbb{Z}\beta+\frac{(4-k)\beta}{8}}^-.\label{7.39}(k=1,3,j=0,1)
      \end{equation}
      \begin{eqnarray}
          (23).(V_{\mathbb{Z}\beta}^-)^i\boxtimes V_{\mathbb{Z}\gamma+\frac{r\gamma}{18}}=V_{\mathbb{Z}\gamma+\frac{\gamma}{18}}\oplus V_{\mathbb{Z}\gamma+\frac{5\gamma}{18}}\oplus V_{\mathbb{Z}\gamma+\frac{7\gamma}{18}}, \label{7.40}\mbox{ for r is odd}.
      \end{eqnarray}
      \begin{eqnarray}
          (V_{\mathbb{Z}\beta}^-)^i\boxtimes V_{\mathbb{Z}\gamma+\frac{r\gamma}{18}}=V_{\mathbb{Z}\gamma+\frac{2\gamma}{18}}\oplus V_{\mathbb{Z}\gamma+\frac{4\gamma}{18}}\oplus V_{\mathbb{Z}\gamma+\frac{8\gamma}{18}}, \label{7.41}\mbox{ for r is even}.
      \end{eqnarray}
      \begin{equation}
\begin{aligned}
& (V_{\mathbb{Z}\beta}^-)^+
  \boxtimes V_{\mathbb{Z}\zeta+\frac{\zeta}{32}}
 = V_{\mathbb{Z}\zeta+\frac{\zeta}{32}}
 \oplus V_{\mathbb{Z}\zeta+\frac{7\zeta}{32}}
 \oplus V_{\mathbb{Z}\zeta+\frac{9\zeta}{32}},\\
& (V_{\mathbb{Z}\beta}^-)^+
  \boxtimes V_{\mathbb{Z}\zeta+\frac{3\zeta}{32}}
 = V_{\mathbb{Z}\zeta+\frac{3\zeta}{32}}
 \oplus V_{\mathbb{Z}\zeta+\frac{5\zeta}{32}}
 \oplus V_{\mathbb{Z}\zeta+\frac{11\zeta}{32}},\\
& (V_{\mathbb{Z}\beta}^-)^+
  \boxtimes V_{\mathbb{Z}\zeta+\frac{5\zeta}{32}}
 = V_{\mathbb{Z}\zeta+\frac{3\zeta}{32}}
 \oplus V_{\mathbb{Z}\zeta+\frac{5\zeta}{32}}
 \oplus V_{\mathbb{Z}\zeta+\frac{13\zeta}{32}},\\
& (V_{\mathbb{Z}\beta}^-)^+
  \boxtimes V_{\mathbb{Z}\zeta+\frac{7\zeta}{32}}
 = V_{\mathbb{Z}\zeta+\frac{\zeta}{32}}
 \oplus V_{\mathbb{Z}\zeta+\frac{7\zeta}{32}}
 \oplus V_{\mathbb{Z}\zeta+\frac{15\zeta}{32}},\\
& (V_{\mathbb{Z}\beta}^-)^+
  \boxtimes V_{\mathbb{Z}\zeta+\frac{9\zeta}{32}}
 = V_{\mathbb{Z}\zeta+\frac{\zeta}{32}}
 \oplus V_{\mathbb{Z}\zeta+\frac{9\zeta}{32}}
 \oplus V_{\mathbb{Z}\zeta+\frac{15\zeta}{32}},\\
& (V_{\mathbb{Z}\beta}^-)^+
  \boxtimes V_{\mathbb{Z}\zeta+\frac{11\zeta}{32}}
 = V_{\mathbb{Z}\zeta+\frac{3\zeta}{32}}
 \oplus V_{\mathbb{Z}\zeta+\frac{11\zeta}{32}}
 \oplus V_{\mathbb{Z}\zeta+\frac{13\zeta}{32}},\\
& (V_{\mathbb{Z}\beta}^-)^+
  \boxtimes V_{\mathbb{Z}\zeta+\frac{13\zeta}{32}}
 = V_{\mathbb{Z}\zeta+\frac{5\zeta}{32}}
 \oplus V_{\mathbb{Z}\zeta+\frac{11\zeta}{32}}
 \oplus V_{\mathbb{Z}\zeta+\frac{13\zeta}{32}},\\
& (V_{\mathbb{Z}\beta}^-)^+
  \boxtimes V_{\mathbb{Z}\zeta+\frac{15\zeta}{32}}
 = V_{\mathbb{Z}\zeta+\frac{7\zeta}{32}}
 \oplus V_{\mathbb{Z}\zeta+\frac{9\zeta}{32}}
 \oplus V_{\mathbb{Z}\zeta+\frac{15\zeta}{32}},\\[4pt]
& (V_{\mathbb{Z}\beta}^-)^-
  \boxtimes V_{\mathbb{Z}\zeta+\frac{\zeta}{32}}
 = V_{\mathbb{Z}\zeta+\frac{7\zeta}{32}}
 \oplus V_{\mathbb{Z}\zeta+\frac{9\zeta}{32}}
 \oplus V_{\mathbb{Z}\zeta+\frac{15\zeta}{32}},\\
& (V_{\mathbb{Z}\beta}^-)^-
  \boxtimes V_{\mathbb{Z}\zeta+\frac{3\zeta}{32}}
 = V_{\mathbb{Z}\zeta+\frac{5\zeta}{32}}
 \oplus V_{\mathbb{Z}\zeta+\frac{11\zeta}{32}}
 \oplus V_{\mathbb{Z}\zeta+\frac{13\zeta}{32}},\\
& (V_{\mathbb{Z}\beta}^-)^-
  \boxtimes V_{\mathbb{Z}\zeta+\frac{5\zeta}{32}}
 = V_{\mathbb{Z}\zeta+\frac{3\zeta}{32}}
 \oplus V_{\mathbb{Z}\zeta+\frac{11\zeta}{32}}
 \oplus V_{\mathbb{Z}\zeta+\frac{13\zeta}{32}},\\
& (V_{\mathbb{Z}\beta}^-)^-
  \boxtimes V_{\mathbb{Z}\zeta+\frac{7\zeta}{32}}
 = V_{\mathbb{Z}\zeta+\frac{\zeta}{32}}
 \oplus V_{\mathbb{Z}\zeta+\frac{9\zeta}{32}}
 \oplus V_{\mathbb{Z}\zeta+\frac{15\zeta}{32}},\\
& (V_{\mathbb{Z}\beta}^-)^-
  \boxtimes V_{\mathbb{Z}\zeta+\frac{9\zeta}{32}}
 = V_{\mathbb{Z}\zeta+\frac{\zeta}{32}}
 \oplus V_{\mathbb{Z}\zeta+\frac{7\zeta}{32}}
 \oplus V_{\mathbb{Z}\zeta+\frac{15\zeta}{32}},\\
& (V_{\mathbb{Z}\beta}^-)^-
  \boxtimes V_{\mathbb{Z}\zeta+\frac{11\zeta}{32}}
 = V_{\mathbb{Z}\zeta+\frac{3\zeta}{32}}
 \oplus V_{\mathbb{Z}\zeta+\frac{5\zeta}{32}}
 \oplus V_{\mathbb{Z}\zeta+\frac{13\zeta}{32}},\\
& (V_{\mathbb{Z}\beta}^-)^-
  \boxtimes V_{\mathbb{Z}\zeta+\frac{13\zeta}{32}}
 = V_{\mathbb{Z}\zeta+\frac{3\zeta}{32}}
 \oplus V_{\mathbb{Z}\zeta+\frac{5\zeta}{32}}
 \oplus V_{\mathbb{Z}\zeta+\frac{11\zeta}{32}},\\
& (V_{\mathbb{Z}\beta}^-)^-
  \boxtimes V_{\mathbb{Z}\zeta+\frac{15\zeta}{32}}
 = V_{\mathbb{Z}\zeta+\frac{\zeta}{32}}
 \oplus V_{\mathbb{Z}\zeta+\frac{7\zeta}{32}}
 \oplus V_{\mathbb{Z}\zeta+\frac{9\zeta}{32}}.
\end{aligned}
\label{7.44}
\end{equation}

      \begin{equation}
         (25). (V_{\mathbb{Z}\beta+\frac{\beta}{4}}^0)^i\boxtimes (V_{\mathbb{Z}\beta+\frac{\beta}{4}}^0)^j=((V_{\mathbb{Z}\beta}^+)^0)^{\overline{i+j}}\oplus (V_{\mathbb{Z}\beta}^-)^{\overline{i+j}}. \label{7.45}(i,j=0,1)
      \end{equation}
      \begin{equation}
          (26).(V_{\mathbb{Z}\beta+\frac{\beta}{4}}^0)^i\boxtimes V_{\mathbb{Z}\beta+\frac{\beta}{4}}^1=  (V_{\mathbb{Z}\beta}^+)^1\oplus(V_{\mathbb{Z}\beta}^-)^+\oplus (V_{\mathbb{Z}\beta}^-)^-.\label{7.46} (i=0,1)
      \end{equation}
      \begin{align}
(27).\quad 
&(V_{\mathbb{Z}\beta+\frac{\beta}{4}}^0)^+
 \boxtimes (V_{\mathbb{Z}\beta+\frac{\beta}{8}})^+
 =
 V_{\mathbb{Z}\beta+\frac{3\beta}{8}}^+
 \oplus
 V_{\mathbb{Z}\beta+\frac{\beta}{8}}^+,
\nonumber\\
&
(V_{\mathbb{Z}\beta+\frac{\beta}{4}}^0)^+
 \boxtimes (V_{\mathbb{Z}\beta+\frac{\beta}{8}})^-
 =
 V_{\mathbb{Z}\beta+\frac{3\beta}{8}}^-
 \oplus
 V_{\mathbb{Z}\beta+\frac{\beta}{8}}^-,
\nonumber\\
&
(V_{\mathbb{Z}\beta+\frac{\beta}{4}}^0)^+
 \boxtimes (V_{\mathbb{Z}\beta+\frac{3\beta}{8}})^+
 =
 V_{\mathbb{Z}\beta+\frac{3\beta}{8}}^-
 \oplus
 V_{\mathbb{Z}\beta+\frac{\beta}{8}}^+,
\nonumber\\
&
(V_{\mathbb{Z}\beta+\frac{\beta}{4}}^0)^+
 \boxtimes (V_{\mathbb{Z}\beta+\frac{3\beta}{8}})^-
 =
 V_{\mathbb{Z}\beta+\frac{3\beta}{8}}^+
 \oplus
 V_{\mathbb{Z}\beta+\frac{\beta}{8}}^-,
\nonumber\\
&
(V_{\mathbb{Z}\beta+\frac{\beta}{4}}^0)^-
 \boxtimes (V_{\mathbb{Z}\beta+\frac{\beta}{8}})^+
 =
 V_{\mathbb{Z}\beta+\frac{3\beta}{8}}^-
 \oplus
 V_{\mathbb{Z}\beta+\frac{\beta}{8}}^-,
\nonumber\\
&
(V_{\mathbb{Z}\beta+\frac{\beta}{4}}^0)^-
 \boxtimes (V_{\mathbb{Z}\beta+\frac{\beta}{8}})^-
 =
 V_{\mathbb{Z}\beta+\frac{3\beta}{8}}^+
 \oplus
 V_{\mathbb{Z}\beta+\frac{\beta}{8}}^+,
\nonumber\\
&
(V_{\mathbb{Z}\beta+\frac{\beta}{4}}^0)^-
 \boxtimes (V_{\mathbb{Z}\beta+\frac{3\beta}{8}})^+
 =
 V_{\mathbb{Z}\beta+\frac{3\beta}{8}}^+
 \oplus
 V_{\mathbb{Z}\beta+\frac{\beta}{8}}^-,
\nonumber\\
&
(V_{\mathbb{Z}\beta+\frac{\beta}{4}}^0)^-
 \boxtimes (V_{\mathbb{Z}\beta+\frac{3\beta}{8}})^-
 =
 V_{\mathbb{Z}\beta+\frac{3\beta}{8}}^-
 \oplus
 V_{\mathbb{Z}\beta+\frac{\beta}{8}}^+.
\label{7.46a}
\end{align}
      \begin{align}
(28).\quad
&(V_{\mathbb{Z}\beta+\frac{\beta}{4}}^0)^i
 \boxtimes V_{\mathbb{Z}\gamma+\frac{\gamma}{18}}
 =
 V_{\mathbb{Z}\gamma+\frac{2\gamma}{18}}
 \oplus
 V_{\mathbb{Z}\gamma+\frac{4\gamma}{18}},
\nonumber\\
&
(V_{\mathbb{Z}\beta+\frac{\beta}{4}}^0)^i
 \boxtimes V_{\mathbb{Z}\gamma+\frac{2\gamma}{18}}
 =
 V_{\mathbb{Z}\gamma+\frac{\gamma}{18}}
 \oplus
 V_{\mathbb{Z}\gamma+\frac{5\gamma}{18}},
\nonumber\\
&
(V_{\mathbb{Z}\beta+\frac{\beta}{4}}^0)^i
 \boxtimes V_{\mathbb{Z}\gamma+\frac{4\gamma}{18}}
 =
 V_{\mathbb{Z}\gamma+\frac{\gamma}{18}}
 \oplus
 V_{\mathbb{Z}\gamma+\frac{7\gamma}{18}},
\nonumber\\
&
(V_{\mathbb{Z}\beta+\frac{\beta}{4}}^0)^i
 \boxtimes V_{\mathbb{Z}\gamma+\frac{5\gamma}{18}}
 =
 V_{\mathbb{Z}\gamma+\frac{2\gamma}{18}}
 \oplus
 V_{\mathbb{Z}\gamma+\frac{8\gamma}{18}},
\nonumber\\
&
(V_{\mathbb{Z}\beta+\frac{\beta}{4}}^0)^i
 \boxtimes V_{\mathbb{Z}\gamma+\frac{7\gamma}{18}}
 =
 V_{\mathbb{Z}\gamma+\frac{4\gamma}{18}}
 \oplus
 V_{\mathbb{Z}\gamma+\frac{8\gamma}{18}},
\nonumber\\
&
(V_{\mathbb{Z}\beta+\frac{\beta}{4}}^0)^i
 \boxtimes V_{\mathbb{Z}\gamma+\frac{8\gamma}{18}}
 =
 V_{\mathbb{Z}\gamma+\frac{5\gamma}{18}}
 \oplus
 V_{\mathbb{Z}\gamma+\frac{7\gamma}{18}}.
\label{7.47}
\end{align}
      \begin{align}
(29).\quad
&(V_{\mathbb{Z}\beta+\frac{\beta}{4}}^0)^+
 \boxtimes V_{\mathbb{Z}\zeta+\frac{\zeta}{32}}
 =
 V_{\mathbb{Z}\zeta+\frac{3\zeta}{32}}
 \oplus
 V_{\mathbb{Z}\zeta+\frac{5\zeta}{32}},
\nonumber\\
&
(V_{\mathbb{Z}\beta+\frac{\beta}{4}}^0)^+
 \boxtimes V_{\mathbb{Z}\zeta+\frac{3\zeta}{32}}
 =
 V_{\mathbb{Z}\zeta+\frac{\zeta}{32}}
 \oplus
 V_{\mathbb{Z}\zeta+\frac{7\zeta}{32}},
\nonumber\\
&
(V_{\mathbb{Z}\beta+\frac{\beta}{4}}^0)^+
 \boxtimes V_{\mathbb{Z}\zeta+\frac{5\zeta}{32}}
 =
 V_{\mathbb{Z}\zeta+\frac{\zeta}{32}}
 \oplus
 V_{\mathbb{Z}\zeta+\frac{9\zeta}{32}},
\nonumber\\
&
(V_{\mathbb{Z}\beta+\frac{\beta}{4}}^0)^+
 \boxtimes V_{\mathbb{Z}\zeta+\frac{7\zeta}{32}}
 =
 V_{\mathbb{Z}\zeta+\frac{3\zeta}{32}}
 \oplus
 V_{\mathbb{Z}\zeta+\frac{11\zeta}{32}},
\nonumber\\
&
(V_{\mathbb{Z}\beta+\frac{\beta}{4}}^0)^+
 \boxtimes V_{\mathbb{Z}\zeta+\frac{9\zeta}{32}}
 =
 V_{\mathbb{Z}\zeta+\frac{5\zeta}{32}}
 \oplus
 V_{\mathbb{Z}\zeta+\frac{13\zeta}{32}},
\nonumber\\
&
(V_{\mathbb{Z}\beta+\frac{\beta}{4}}^0)^+
 \boxtimes V_{\mathbb{Z}\zeta+\frac{11\zeta}{32}}
 =
 V_{\mathbb{Z}\zeta+\frac{7\zeta}{32}}
 \oplus
 V_{\mathbb{Z}\zeta+\frac{15\zeta}{32}},
\nonumber\\
&
(V_{\mathbb{Z}\beta+\frac{\beta}{4}}^0)^+
 \boxtimes V_{\mathbb{Z}\zeta+\frac{13\zeta}{32}}
 =
 V_{\mathbb{Z}\zeta+\frac{9\zeta}{32}}
 \oplus
 V_{\mathbb{Z}\zeta+\frac{15\zeta}{32}},
\nonumber\\
&
(V_{\mathbb{Z}\beta+\frac{\beta}{4}}^0)^+
 \boxtimes V_{\mathbb{Z}\zeta+\frac{15\zeta}{32}}
 =
 V_{\mathbb{Z}\zeta+\frac{11\zeta}{32}}
 \oplus
 V_{\mathbb{Z}\zeta+\frac{13\zeta}{32}},
\nonumber\\
&
(V_{\mathbb{Z}\beta+\frac{\beta}{4}}^0)^-
 \boxtimes V_{\mathbb{Z}\zeta+\frac{\zeta}{32}}
 =
 V_{\mathbb{Z}\zeta+\frac{11\zeta}{32}}
 \oplus
 V_{\mathbb{Z}\zeta+\frac{13\zeta}{32}},
\nonumber\\
&
(V_{\mathbb{Z}\beta+\frac{\beta}{4}}^0)^-
 \boxtimes V_{\mathbb{Z}\zeta+\frac{3\zeta}{32}}
 =
 V_{\mathbb{Z}\zeta+\frac{9\zeta}{32}}
 \oplus
 V_{\mathbb{Z}\zeta+\frac{15\zeta}{32}},
\nonumber\\
&
(V_{\mathbb{Z}\beta+\frac{\beta}{4}}^0)^-
 \boxtimes V_{\mathbb{Z}\zeta+\frac{5\zeta}{32}}
 =
 V_{\mathbb{Z}\zeta+\frac{7\zeta}{32}}
 \oplus
 V_{\mathbb{Z}\zeta+\frac{15\zeta}{32}},
\nonumber\\
&
(V_{\mathbb{Z}\beta+\frac{\beta}{4}}^0)^-
 \boxtimes V_{\mathbb{Z}\zeta+\frac{7\zeta}{32}}
 =
 V_{\mathbb{Z}\zeta+\frac{5\zeta}{32}}
 \oplus
 V_{\mathbb{Z}\zeta+\frac{13\zeta}{32}},
\nonumber\\
&
(V_{\mathbb{Z}\beta+\frac{\beta}{4}}^0)^-
 \boxtimes V_{\mathbb{Z}\zeta+\frac{9\zeta}{32}}
 =
 V_{\mathbb{Z}\zeta+\frac{3\zeta}{32}}
 \oplus
 V_{\mathbb{Z}\zeta+\frac{11\zeta}{32}},
\nonumber\\
&
(V_{\mathbb{Z}\beta+\frac{\beta}{4}}^0)^-
 \boxtimes V_{\mathbb{Z}\zeta+\frac{11\zeta}{32}}
 =
 V_{\mathbb{Z}\zeta+\frac{\zeta}{32}}
 \oplus
 V_{\mathbb{Z}\zeta+\frac{9\zeta}{32}},
\nonumber\\
&
(V_{\mathbb{Z}\beta+\frac{\beta}{4}}^0)^-
 \boxtimes V_{\mathbb{Z}\zeta+\frac{13\zeta}{32}}
 =
 V_{\mathbb{Z}\zeta+\frac{\zeta}{32}}
 \oplus
 V_{\mathbb{Z}\zeta+\frac{7\zeta}{32}},
\nonumber\\
&
(V_{\mathbb{Z}\beta+\frac{\beta}{4}}^0)^-
 \boxtimes V_{\mathbb{Z}\zeta+\frac{15\zeta}{32}}
 =
 V_{\mathbb{Z}\zeta+\frac{3\zeta}{32}}
 \oplus
 V_{\mathbb{Z}\zeta+\frac{5\zeta}{32}}.
\label{7.48}
\end{align}
      \begin{equation}
          (30).(V_{\mathbb{Z}\beta+\frac{\beta}{4}}^0)^i\boxtimes V_{\mathbb{Z}\zeta}^{T_2,j}=V_{\mathbb{Z}\zeta}^{T_2,+}\oplus V_{\mathbb{Z}\zeta}^{T_2,-}. \quad(i=0,1,j=0,1)\label{7.49}
      \end{equation}
      \begin{equation}
          (31).V_{\mathbb{Z}\beta+\frac{\beta}{4}}^1\boxtimes V_{\mathbb{Z}\beta+\frac{\beta}{4}}^1=((V_{\mathbb{Z}\beta}^+)^0)^+\oplus ((V_{\mathbb{Z}\beta}^+)^0)^-\oplus (V_{\mathbb{Z}\beta}^+)^1\oplus 2(V_{\mathbb{Z}\beta}^-)^+\oplus 2(V_{\mathbb{Z}\beta}^-)^-.\label{7.50}
      \end{equation}
      \begin{equation}
          (32).V_{\mathbb{Z}\beta+\frac{\beta}{4}}^1\boxtimes (V_{\mathbb{Z}\beta+\frac{k\beta}{8}})^j=\bigoplus_{k=1,3,s=0,1} (V_{\mathbb{Z}\beta+\frac{k\beta}{8}})^s.\label{7.51}
      \end{equation}
      \begin{eqnarray}
         (33). & &V_{\mathbb{Z}\beta+\frac{\beta}{4}}^1\boxtimes (V_{\mathbb{Z}\gamma+\frac{\gamma}{18}})=V_{\mathbb{Z}\gamma+\frac{2\gamma}{18}}\oplus V_{\mathbb{Z}\gamma+\frac{4\gamma}{18}}\oplus 2V_{\mathbb{Z}\gamma+\frac{8\gamma}{18}}, \nonumber\\
          & & V_{\mathbb{Z}\beta+\frac{\beta}{4}}^1\boxtimes (V_{\mathbb{Z}\gamma+\frac{2\gamma}{18}})=V_{\mathbb{Z}\gamma+\frac{\gamma}{18}}\oplus V_{\mathbb{Z}\gamma+\frac{5\gamma}{18}}\oplus 2V_{\mathbb{Z}\gamma+\frac{7\gamma}{18}}, \nonumber\\
          & &V_{\mathbb{Z}\beta+\frac{\beta}{4}}^1\boxtimes (V_{\mathbb{Z}\gamma+\frac{4\gamma}{18}})=V_{\mathbb{Z}\gamma+\frac{\gamma}{18}}\oplus 2V_{\mathbb{Z}\gamma+\frac{5\gamma}{18}}\oplus V_{\mathbb{Z}\gamma+\frac{7\gamma}{18}}, \nonumber\\
          & &V_{\mathbb{Z}\beta+\frac{\beta}{4}}^1\boxtimes (V_{\mathbb{Z}\gamma+\frac{5\gamma}{18}})=V_{\mathbb{Z}\gamma+\frac{2\gamma}{18}}\oplus 2V_{\mathbb{Z}\gamma+\frac{4\gamma}{18}}\oplus V_{\mathbb{Z}\gamma+\frac{8\gamma}{18}}, \nonumber\\
          & &V_{\mathbb{Z}\beta+\frac{\beta}{4}}^1\boxtimes (V_{\mathbb{Z}\gamma+\frac{7\gamma}{18}})=2V_{\mathbb{Z}\gamma+\frac{2\gamma}{18}}\oplus V_{\mathbb{Z}\gamma+\frac{4\gamma}{18}}\oplus V_{\mathbb{Z}\gamma+\frac{8\gamma}{18}}, \nonumber\\
          & &V_{\mathbb{Z}\beta+\frac{\beta}{4}}^1\boxtimes (V_{\mathbb{Z}\gamma+\frac{8\gamma}{18}})=2V_{\mathbb{Z}\gamma+\frac{\gamma}{18}}\oplus V_{\mathbb{Z}\gamma+\frac{5\gamma}{18}}\oplus V_{\mathbb{Z}\gamma+\frac{7\gamma}{18}}, \nonumber\\\label {7.52}
      \end{eqnarray}
      \begin{align}
(34). V_{\mathbb{Z}\beta+\frac{\beta}{4}}^1 \boxtimes 
V_{\mathbb{Z}\zeta+\frac{s\zeta}{32}}
&=
V_{\mathbb{Z}\zeta+\frac{3\zeta}{32}}
\oplus
V_{\mathbb{Z}\zeta+\frac{5\zeta}{32}}
\oplus
V_{\mathbb{Z}\zeta+\frac{11\zeta}{32}}
\oplus
V_{\mathbb{Z}\zeta+\frac{13\zeta}{32}},
\quad (s=1,7,9,15) \nonumber \\
V_{\mathbb{Z}\beta+\frac{\beta}{4}}^1 \boxtimes 
V_{\mathbb{Z}\zeta+\frac{s\zeta}{32}}&=
V_{\mathbb{Z}\zeta+\frac{\zeta}{32}}
\oplus
V_{\mathbb{Z}\zeta+\frac{7\zeta}{32}}
\oplus
V_{\mathbb{Z}\zeta+\frac{9\zeta}{32}}
\oplus
V_{\mathbb{Z}\zeta+\frac{15\zeta}{32}},
\quad (s=3,5,11,13).
\label{7.53}
\end{align}

\begin{eqnarray}
     (35).(V_{\mathbb{Z}\beta+\frac{i\beta}{8}})^j\boxtimes (V_{\mathbb{Z}\beta+\frac{k\beta}{8}})^l&=&(V_{\mathbb{Z}\beta+\frac{\beta}{4}}^{\overline{j+l}})\oplus V_{\mathbb{Z}\beta+\frac{\beta}{4}}^1\oplus ((V_{\mathbb{Z}\beta}^+)^0)^{\overline{j+l}}\oplus (V_{\mathbb{Z}\beta}^+)^1\oplus (V_{\mathbb{Z}\beta}^-)^{\overline{j+l}}\oplus V_{\mathbb{Z}\beta+\frac{\beta}{8}}^{\pm}\nonumber \\
    & &\oplus V_{\mathbb{Z}\beta+\frac{3\beta}{8}}^{\pm}, \mbox{ for i=k=1,j,l=0,1}.
    \label{7.54a}
\end{eqnarray}
\begin{eqnarray}
    (V_{\mathbb{Z}\beta+\frac{i\beta}{8}})^j\boxtimes (V_{\mathbb{Z}\beta+\frac{k\beta}{8}})^l&=&(V_{\mathbb{Z}\beta+\frac{\beta}{4}}^{\overline{j+l+1}})\oplus V_{\mathbb{Z}\beta+\frac{\beta}{4}}^1\oplus ((V_{\mathbb{Z}\beta}^+)^0)^{\overline{j+l}}\oplus (V_{\mathbb{Z}\beta}^+)^1\oplus (V_{\mathbb{Z}\beta}^-)^{\overline{j+l}}\oplus V_{\mathbb{Z}\beta+\frac{\beta}{8}}^{+}\nonumber \\
    & &\oplus V_{\mathbb{Z}\beta+\frac{\beta}{8}}^{-}\oplus V_{\mathbb{Z}\beta+\frac{3\beta}{8}}^{+}\oplus V_{\mathbb{Z}\beta+\frac{3\beta}{8}}^{-}, \mbox{ for i=k=3,j,l=0,1}.\label{7.55}
\end{eqnarray}
\begin{eqnarray}
    (V_{\mathbb{Z}\beta+\frac{i\beta}{8}})^j\boxtimes (V_{\mathbb{Z}\beta+\frac{k\beta}{8}})^l&=&(V_{\mathbb{Z}\beta+\frac{\beta}{4}}^{\overline{j+l}})\oplus V_{\mathbb{Z}\beta+\frac{\beta}{4}}^1\oplus (V_{\mathbb{Z}\beta}^-)^+\oplus (V_{\mathbb{Z}\beta}^-)^-\oplus V_{\mathbb{Z}\beta+\frac{\beta}{8}}^{+}\nonumber \\
    & &\oplus V_{\mathbb{Z}\beta+\frac{\beta}{8}}^{-} \oplus V_{\mathbb{Z}\beta+\frac{3\beta}{8}}^{+}\oplus V_{\mathbb{Z}\beta+\frac{3\beta}{8}}^{-}, \mbox{ for } i\neq k,j,l=0,1. \label{7.56}
\end{eqnarray}
\begin{eqnarray}
    (36). W_{\delta,1}^k\boxtimes W_{\delta,1}^{l}&=&W_{\delta,1}^0\oplus W_{\delta,1}^1\oplus W_{\delta,1}^2\oplus W_{\delta,2}^{\overline{\overline{k+l}}}\oplus ((V_{\mathbb{Z}\beta}^+)^0)^+\oplus ((V_{\mathbb{Z}\beta}^+)^0)^-\oplus (V_{\mathbb{Z}\beta}^-)^+\oplus (V_{\mathbb{Z}\beta}^-)^-\nonumber\\ 
    & & \oplus (V_{\mathbb{Z}\beta+\frac{\beta}{8}})^+\oplus (V_{\mathbb{Z}\beta+\frac{\beta}{8}})^-\oplus (V_{\mathbb{Z}\beta+\frac{3\beta}{8}})^+\oplus (V_{\mathbb{Z}\beta+\frac{3\beta}{8}})^-.\mbox{ for }k=l,k,l=0,1,2.\nonumber\\ \label{7.57}
\end{eqnarray}
\begin{eqnarray}
    W_{\delta,1}^k\boxtimes W_{\delta,1}^{l}&=&W_{\delta,1}^0\oplus W_{\delta,1}^1\oplus W_{\delta,1}^2\oplus W_{\delta,2}^{\overline{\overline{k+l}}}\oplus (V_{\mathbb{Z}\beta}^+)^1\oplus (V_{\mathbb{Z}\beta}^-)^+\oplus (V_{\mathbb{Z}\beta}^-)^-\nonumber\\ 
    & & \oplus (V_{\mathbb{Z}\beta+\frac{\beta}{8}})^+\oplus (V_{\mathbb{Z}\beta+\frac{\beta}{8}})^-\oplus (V_{\mathbb{Z}\beta+\frac{3\beta}{8}})^+\oplus (V_{\mathbb{Z}\beta+\frac{3\beta}{8}})^-.\mbox{ for }k\neq l,k,l=0,1,2.\nonumber\\ \label{7.58}
\end{eqnarray}
\begin{eqnarray}
    W_{\delta,2}^k\boxtimes W_{\delta,2}^{l}&=&W_{\delta,1}^0\oplus W_{\delta,1}^1\oplus W_{\delta,1}^2\oplus W_{\delta,2}^{\overline{\overline{1-k-l}}}\oplus ((V_{\mathbb{Z}\beta}^+)^0)^+\oplus ((V_{\mathbb{Z}\beta}^+)^0)^-\oplus (V_{\mathbb{Z}\beta}^-)^+\oplus (V_{\mathbb{Z}\beta}^-)^-\nonumber\\ 
    & & \oplus (V_{\mathbb{Z}\beta+\frac{\beta}{8}})^+\oplus (V_{\mathbb{Z}\beta+\frac{\beta}{8}})^-\oplus (V_{\mathbb{Z}\beta+\frac{3\beta}{8}})^+\oplus (V_{\mathbb{Z}\beta+\frac{3\beta}{8}})^-.\mbox{ for }k=l,k,l=0,1,2.\nonumber\\ \label{7.59}
\end{eqnarray}
\begin{eqnarray}
    W_{\delta,2}^k\boxtimes W_{\delta,2}^{l}&=&W_{\delta,1}^0\oplus W_{\delta,1}^1\oplus W_{\delta,1}^2\oplus W_{\delta,2}^{\overline{\overline{1-k-l}}}\oplus (V_{\mathbb{Z}\beta}^+)^1\oplus (V_{\mathbb{Z}\beta}^-)^+\oplus (V_{\mathbb{Z}\beta}^-)^-\nonumber\\ 
    & & \oplus (V_{\mathbb{Z}\beta+\frac{\beta}{8}})^+\oplus (V_{\mathbb{Z}\beta+\frac{\beta}{8}})^-\oplus (V_{\mathbb{Z}\beta+\frac{3\beta}{8}})^+\oplus (V_{\mathbb{Z}\beta+\frac{3\beta}{8}})^-.\mbox{ for }k\neq l,k,l=0,1,2.\nonumber\\ \label{7.60}
\end{eqnarray}
\begin{eqnarray}
    V_{\mathbb{Z}\gamma+\frac{\gamma}{18}}\boxtimes V_{\mathbb{Z}\gamma+\frac{2\gamma}{18}}&=&V_{\mathbb{Z}\gamma+\frac{\gamma}{18}}\oplus V_{\mathbb{Z}\gamma+\frac{2\gamma}{18}}\oplus V_{\mathbb{Z}\gamma+\frac{4\gamma}{18}}\oplus V_{\mathbb{Z}\gamma+\frac{8\gamma}{18}}\oplus (V_{\mathbb{Z}\beta+\frac{\beta}{4}}^0)^+\oplus (V_{\mathbb{Z}\beta+\frac{\beta}{4}}^0)^-\oplus V_{\mathbb{Z}\beta+\frac{\beta}{4}}^1\nonumber\\ 
    & &\oplus (V_{\mathbb{Z}\beta+\frac{\beta}{8}})^+\oplus (V_{\mathbb{Z}\beta+\frac{\beta}{8}})^-\oplus (V_{\mathbb{Z}\beta+\frac{3\beta}{8}})^+\oplus (V_{\mathbb{Z}\beta+\frac{3\beta}{8}})^-.\label{7.61}
\end{eqnarray}
\begin{eqnarray}
    V_{\mathbb{Z}\gamma+\frac{\gamma}{18}}\boxtimes V_{\mathbb{Z}\gamma+\frac{4\gamma}{18}}&=&V_{\mathbb{Z}\gamma+\frac{2\gamma}{18}}\oplus V_{\mathbb{Z}\gamma+\frac{4\gamma}{18}}\oplus V_{\mathbb{Z}\gamma+\frac{5\gamma}{18}}\oplus V_{\mathbb{Z}\gamma+\frac{8\gamma}{18}}\oplus (V_{\mathbb{Z}\beta+\frac{\beta}{4}}^0)^+\oplus (V_{\mathbb{Z}\beta+\frac{\beta}{4}}^0)^-\oplus V_{\mathbb{Z}\beta+\frac{\beta}{4}}^1\nonumber\\ 
    & &\oplus (V_{\mathbb{Z}\beta+\frac{\beta}{8}})^+\oplus (V_{\mathbb{Z}\beta+\frac{\beta}{8}})^-\oplus (V_{\mathbb{Z}\beta+\frac{3\beta}{8}})^+\oplus (V_{\mathbb{Z}\beta+\frac{3\beta}{8}})^-.\label{7.62}
\end{eqnarray}
\begin{eqnarray}
    V_{\mathbb{Z}\gamma+\frac{\gamma}{18}}\boxtimes V_{\mathbb{Z}\gamma+\frac{8\gamma}{18}}&=&V_{\mathbb{Z}\gamma+\frac{2\gamma}{18}}\oplus V_{\mathbb{Z}\gamma+\frac{4\gamma}{18}}\oplus V_{\mathbb{Z}\gamma+\frac{7\gamma}{18}}\oplus V_{\mathbb{Z}\gamma+\frac{8\gamma}{18}}\oplus 2V_{\mathbb{Z}\beta+\frac{\beta}{4}}^1\nonumber\\ 
    & &\oplus (V_{\mathbb{Z}\beta+\frac{\beta}{8}})^+\oplus (V_{\mathbb{Z}\beta+\frac{\beta}{8}})^-\oplus (V_{\mathbb{Z}\beta+\frac{3\beta}{8}})^+\oplus (V_{\mathbb{Z}\beta+\frac{3\beta}{8}})^-.\label{7.63}
\end{eqnarray}
\begin{eqnarray}
    V_{\mathbb{Z}\gamma+\frac{2\gamma}{18}}\boxtimes V_{\mathbb{Z}\gamma+\frac{5\gamma}{18}}&=&V_{\mathbb{Z}\gamma+\frac{2\gamma}{18}}\oplus V_{\mathbb{Z}\gamma+\frac{4\gamma}{18}}\oplus V_{\mathbb{Z}\gamma+\frac{7\gamma}{18}}\oplus V_{\mathbb{Z}\gamma+\frac{8\gamma}{18}}\oplus (V_{\mathbb{Z}\beta+\frac{\beta}{4}}^0)^+\oplus (V_{\mathbb{Z}\beta+\frac{\beta}{4}}^0)^-\oplus V_{\mathbb{Z}\beta+\frac{\beta}{4}}^1\nonumber\\ 
    & &\oplus (V_{\mathbb{Z}\beta+\frac{\beta}{8}})^+\oplus (V_{\mathbb{Z}\beta+\frac{\beta}{8}})^-\oplus (V_{\mathbb{Z}\beta+\frac{3\beta}{8}})^+\oplus (V_{\mathbb{Z}\beta+\frac{3\beta}{8}})^-.\label{7.64}
\end{eqnarray}
\begin{eqnarray}
    V_{\mathbb{Z}\gamma+\frac{4\gamma}{18}}\boxtimes V_{\mathbb{Z}\gamma+\frac{5\gamma}{18}}&=&V_{\mathbb{Z}\gamma+\frac{\gamma}{18}}\oplus V_{\mathbb{Z}\gamma+\frac{2\gamma}{18}}\oplus V_{\mathbb{Z}\gamma+\frac{4\gamma}{18}}\oplus V_{\mathbb{Z}\gamma+\frac{8\gamma}{18}}\oplus 2V_{\mathbb{Z}\beta+\frac{\beta}{4}}^1\nonumber\\ 
    & &\oplus (V_{\mathbb{Z}\beta+\frac{\beta}{8}})^+\oplus (V_{\mathbb{Z}\beta+\frac{\beta}{8}})^-\oplus (V_{\mathbb{Z}\beta+\frac{3\beta}{8}})^+\oplus (V_{\mathbb{Z}\beta+\frac{3\beta}{8}})^-.\label{7.65}
\end{eqnarray}

\begin{eqnarray}
    V_{\mathbb{Z}\gamma+\frac{5\gamma}{18}}\boxtimes V_{\mathbb{Z}\gamma+\frac{8\gamma}{18}}&=&V_{\mathbb{Z}\gamma+\frac{2\gamma}{18}}\oplus V_{\mathbb{Z}\gamma+\frac{4\gamma}{18}}\oplus V_{\mathbb{Z}\gamma+\frac{5\gamma}{18}}\oplus V_{\mathbb{Z}\gamma+\frac{8\gamma}{18}}\oplus (V_{\mathbb{Z}\beta+\frac{\beta}{4}}^0)^+\oplus (V_{\mathbb{Z}\beta+\frac{\beta}{4}}^0)^-\oplus V_{\mathbb{Z}\beta+\frac{\beta}{4}}^1\nonumber\\ 
    & &\oplus (V_{\mathbb{Z}\beta+\frac{\beta}{8}})^+\oplus (V_{\mathbb{Z}\beta+\frac{\beta}{8}})^-\oplus (V_{\mathbb{Z}\beta+\frac{3\beta}{8}})^+\oplus (V_{\mathbb{Z}\beta+\frac{3\beta}{8}})^-.\label{7.66}
\end{eqnarray}
\begin{eqnarray}
    V_{\mathbb{Z}\gamma+\frac{2\gamma}{18}}\boxtimes V_{\mathbb{Z}\gamma+\frac{7\gamma}{18}}&=&V_{\mathbb{Z}\gamma+\frac{2\gamma}{18}}\oplus V_{\mathbb{Z}\gamma+\frac{4\gamma}{18}}\oplus V_{\mathbb{Z}\gamma+\frac{5\gamma}{18}}\oplus V_{\mathbb{Z}\gamma+\frac{8\gamma}{18}}\oplus 2V_{\mathbb{Z}\beta+\frac{\beta}{4}}^1\nonumber\\ 
    & &\oplus (V_{\mathbb{Z}\beta+\frac{\beta}{8}})^+\oplus (V_{\mathbb{Z}\beta+\frac{\beta}{8}})^-\oplus (V_{\mathbb{Z}\beta+\frac{3\beta}{8}})^+\oplus (V_{\mathbb{Z}\beta+\frac{3\beta}{8}})^-.\label{7.67}
\end{eqnarray}
\begin{eqnarray}
    V_{\mathbb{Z}\gamma+\frac{4\gamma}{18}}\boxtimes V_{\mathbb{Z}\gamma+\frac{7\gamma}{18}}&=&V_{\mathbb{Z}\gamma+\frac{2\gamma}{18}}\oplus V_{\mathbb{Z}\gamma+\frac{4\gamma}{18}}\oplus V_{\mathbb{Z}\gamma+\frac{7\gamma}{18}}\oplus V_{\mathbb{Z}\gamma+\frac{8\gamma}{18}}\oplus (V_{\mathbb{Z}\beta+\frac{\beta}{4}}^0)^+\oplus (V_{\mathbb{Z}\beta+\frac{\beta}{4}}^0)^-\oplus V_{\mathbb{Z}\beta+\frac{\beta}{4}}^1\nonumber\\ 
    & &\oplus (V_{\mathbb{Z}\beta+\frac{\beta}{8}})^+\oplus (V_{\mathbb{Z}\beta+\frac{\beta}{8}})^-\oplus (V_{\mathbb{Z}\beta+\frac{3\beta}{8}})^+\oplus (V_{\mathbb{Z}\beta+\frac{3\beta}{8}})^-.\label{7.68}
\end{eqnarray}
\begin{eqnarray}
    V_{\mathbb{Z}\gamma+\frac{7\gamma}{18}}\boxtimes V_{\mathbb{Z}\gamma+\frac{8\gamma}{18}}&=&V_{\mathbb{Z}\gamma+\frac{\gamma}{18}}\oplus V_{\mathbb{Z}\gamma+\frac{2\gamma}{18}}\oplus V_{\mathbb{Z}\gamma+\frac{4\gamma}{18}}\oplus V_{\mathbb{Z}\gamma+\frac{8\gamma}{18}}\oplus (V_{\mathbb{Z}\beta+\frac{\beta}{4}}^0)^+\oplus (V_{\mathbb{Z}\beta+\frac{\beta}{4}}^0)^-\oplus V_{\mathbb{Z}\beta+\frac{\beta}{4}}^1\nonumber\\ 
    & &\oplus (V_{\mathbb{Z}\beta+\frac{\beta}{8}})^+\oplus (V_{\mathbb{Z}\beta+\frac{\beta}{8}})^-\oplus (V_{\mathbb{Z}\beta+\frac{3\beta}{8}})^+\oplus (V_{\mathbb{Z}\beta+\frac{3\beta}{8}})^-.\label{7.69}
\end{eqnarray}
  \end{theorem}
  \begin{proof}
      We begin with the proof of $(1)$--$(10)$. By Theorem~5.4, $((V_{\mathbb{Z}\beta}^+)^0)^i$ is a simple current, and its tensor product with any irreducible module is again an irreducible $V_{L_2}^{S_4}$-module. By Corollary~6.6, we have
\[
((V_{\mathbb{Z}\beta}^+)^0)^+ \subseteq V_{\mathbb{Z}\zeta}
\quad \text{and} \quad
((V_{\mathbb{Z}\beta}^+)^0)^- \subseteq V_{\mathbb{Z}\zeta+\frac{\zeta}{2}}.
\]
We first show that
\[
((V_{\mathbb{Z}\beta}^+)^0)^+ \boxtimes ((V_{\mathbb{Z}\beta}^+)^0)^+
= ((V_{\mathbb{Z}\beta}^+)^0)^+.
\]
By Theorem~4.4, some irreducible $V_{L_2}^{S_4}$-modules appearing in $V_{\mathbb{Z}\zeta}$ must also appear in $((V_{\mathbb{Z}\beta}^+)^0)^+ \boxtimes ((V_{\mathbb{Z}\beta}^+)^0)^+.$ The equality then follows from a comparison of quantum dimensions. The remaining cases in $(1)$ are proved in a similar manner.

The same method applies to the proofs of $(1)$--$(10)$ when combined with Corollary~6.5, Proposition~6.6, and Lemma~6.7.

      \emph{proof of (11)}: By Theorem~4.4 and the fusion rules for $V_{L_2}^{A_4}$ \cite{DJJJY}, $(V_{\mathbb{Z}\beta}^+)^1$ must be contained in $(V_{\mathbb{Z}\beta}^+)^1 \boxtimes (V_{\mathbb{Z}\beta}^+)^1.$
Since contragredient modules preserve both quantum dimensions and conformal weights, it follows that all irreducible $V_{L_2}^{S_4}$-modules are self-dual; that is,
$M \cong M'
\quad \text{for every irreducible } V_{L_2}^{S_4}\text{-module } M.$ 

It is well known $\mbox{dim }I_{((V_{\mathbb{Z}\beta}^+)^0)^+}\left(\begin{array}{c}M\\
((V_{\mathbb{Z}\beta}^+)^0)^+\,M\end{array}\right)=1$ for all irreducible $V_{L_2}^{S_4}$-modules. Thus $\mbox{dim }I_{((V_{\mathbb{Z}\beta}^+)^0)^+}\left(\begin{array}{c}((V_{\mathbb{Z}\beta}^+)^0)^+\\
M\,M\end{array}\right)=1$. That is to say there is one and only one copy of $((V_{\mathbb{Z}\beta}^+)^0)^+$ in $(V_{\mathbb{Z}\beta}^+)^1\boxtimes (V_{\mathbb{Z}\beta}^+)^1$. By counting quantum dimension, we know $((V_{\mathbb{Z}\beta}^+)^0)^-\in (V_{\mathbb{Z}\beta}^+)^1\boxtimes (V_{\mathbb{Z}\beta}^+)^1$ and get (11).

\emph{proof of (12)}: By the fusion rules for $V_{L_2}^{A_4}$, we have
\[
(V_{\mathbb{Z}\beta}^+)^1 \boxtimes_{A_4} V_{\mathbb{Z}\beta}^- = V_{\mathbb{Z}\beta}^-.
\]
Let $I$ be an intertwining operator of type $I_{V_{L_2}^{A_4}}\!\left(
\begin{array}{c}
V_{\mathbb{Z}\beta}^- \\
(V_{\mathbb{Z}\beta}^+)^1 \; V_{\mathbb{Z}\beta}^-
\end{array}
\right).$
By \cite{DL}, the restriction of $I$ yields a nonzero intertwining operator of type
\[
I_{V_{L_2}^{S_4}}\!\left(
\begin{array}{c}
V_{\mathbb{Z}\beta}^- \\
(V_{\mathbb{Z}\beta}^+)^1 \; (V_{\mathbb{Z}\beta}^-)^j
\end{array}
\right),
\qquad j = 0,1.
\]
Let $v \in (V_{\mathbb{Z}\beta}^-)^j$. Then
\[
\langle u_n v \rangle = V_{\mathbb{Z}\beta}^-
\]
for any $u \in (V_{\mathbb{Z}\beta}^+)^1$ and $n \in \mathbb{Z}$, where $u_n v$ denotes the coefficient of $z^{-n-1}$ in $I(u,z)v$. This follows from the fact that $V_{\mathbb{Z}\beta}^-$ is an irreducible $V_{L_2}^{A_4}$-module. Consequently,
$(V_{\mathbb{Z}\beta}^+)^1 \boxtimes (V_{\mathbb{Z}\beta}^-)^j$ must contain both $(V_{\mathbb{Z}\beta}^-)^+$ and $(V_{\mathbb{Z}\beta}^-)^-$. The equality then follows by comparing the quantum dimensions of both sides.

\emph{proof of (13)}: This follows from theorem 4.4 and fusion rules for $V_{L_2}^{A_4}$.

\emph{proof of (14)}: By fusion rules for $V_{L_2}^{A_4}$, we have $(V_{\mathbb{Z}\beta}^+)^1\boxtimes_{A_4} V_{\mathbb{Z}\beta+\frac{\beta}{4}}^1 \cong V_{\mathbb{Z}\beta+\frac{\beta}{4}}^2$ and $V_{\mathbb{Z}\beta+\frac{\beta}{4}}^0\boxtimes_{A_4}V_{\mathbb{Z}\beta+\frac{\beta}{4}}^1\cong (V_{\mathbb{Z}\beta})^-\oplus (V_{\mathbb{Z}\beta}^+)^1.$ From first identity, we know $V_{\mathbb{Z}\beta+\frac{\beta}{4}}^2\cong V_{\mathbb{Z}\beta+\frac{\beta}{4}}^1$ as $V_{L_2}^{S_4}$-modules is in $(V_{\mathbb{Z}\beta}^+)^1\boxtimes V_{\mathbb{Z}\beta+\frac{\beta}{4}}^1$. From the second identity, we know $1=\mbox{dim }I_{V_{L_2}^{A_4}}\left(\begin{array}{c}(V_{\mathbb{Z}\beta}^+)^1\\
V_{\mathbb{Z}\beta+\frac{\beta}{4}}^0\,V_{\mathbb{Z}\beta+\frac{\beta}{4}}^1\end{array}\right)\leq \mbox{dim }I_{V_{L_2}^{S_4}}\left(\begin{array}{c}(V_{\mathbb{Z}\beta}^+)^1\\
(V_{\mathbb{Z}\beta+\frac{\beta}{4}}^0)^\pm\,V_{\mathbb{Z}\beta+\frac{\beta}{4}}^1\end{array}\right)=\mbox{dim }I_{V_{L_2}^{S_4}}\left(\begin{array}{c}(V_{\mathbb{Z}\beta+\frac{\beta}{4}}^0)^\pm\\
(V_{\mathbb{Z}\beta}^+)^1\,V_{\mathbb{Z}\beta+\frac{\beta}{4}}^1\end{array}\right)$. This means $(V_{\mathbb{Z}\beta+\frac{\beta}{4}}^0)^\pm$ must be in $(V_{\mathbb{Z}\beta}^+)^1\boxtimes V_{\mathbb{Z}\beta+\frac{\beta}{4}}^1 $. Equality follows by counting quantum dimensions.

\emph{proof of (15)}:This follows from theorem 4.4, proposition 6.5, corollary 6.6.

\emph{proof of (16)}: We prove (\ref{7.29}); the proofs of the remaining cases are similar. As a $V_{L_2}^{S_4}$-module, we have
$V_{\mathbb{Z}\gamma+\frac{\gamma}{18}} \cong W_{\delta,1}^{0} \cong W_{\delta^2,1}^{0}$ by \cite{WZ},\cite{DJJJY}. By Theorem~4.4 and the fusion rules for $V_{L_2}^{A_4}$-module, we see that
$V_{\mathbb{Z}\gamma+\frac{5\gamma}{18}}, \;
V_{\mathbb{Z}\gamma+\frac{7\gamma}{18}}
\subseteq
(V_{\mathbb{Z}\beta}^+)^1 \boxtimes V_{\mathbb{Z}\gamma+\frac{\gamma}{18}}.$ The equality then follows from a comparison of quantum dimensions.

\emph{proof of (17)}: By Corollary~6.6, $(V_{\mathbb{Z}\beta}^+)^1$ is contained in both $V_{\mathbb{Z}\zeta}$ and $V_{\mathbb{Z}\zeta+\frac{\zeta}{2}}$. Using the fusion rules for $V_{\mathbb{Z}\zeta}$-modules, together with Theorem~4.4, quantum dimension arguments, and the fact that $V_{\mathbb{Z}\zeta+\frac{(16+s)\zeta}{32}}
\cong
V_{\mathbb{Z}\zeta+\frac{(16-s)\zeta}{32}}$ as $V_{L_2}^{S_4}$-modules, we obtain $(17)$.

\emph{proof of (18)}: This follows from $\epsilon$-identification in proposition 6.7 and technique used in proof of (17).

  \emph{proof of (19)}: By the fusion rules for $V_{L_2}^{A_4}$, we have
\[
V_{\mathbb{Z}\beta}^- \boxtimes_{A_4} V_{\mathbb{Z}\beta}^-
=
(V_{\mathbb{Z}\beta}^+)^0
\oplus
(V_{\mathbb{Z}\beta}^+)^1
\oplus
(V_{\mathbb{Z}\beta}^+)^2
\oplus
2V_{\mathbb{Z}\beta}^- .
\]
By Theorem~4.4, the fusion product
$(V_{\mathbb{Z}\beta}^-)^i \boxtimes (V_{\mathbb{Z}\beta}^-)^j$ must contain at least one irreducible module from $(V_{\mathbb{Z}\beta}^+)^0$ and at least one copy of $(V_{\mathbb{Z}\beta}^+)^1$. Since all irreducible $V_{L_2}^{S_4}$-modules are self-dual, we have
\[
\bigl((V_{\mathbb{Z}\beta}^+)^0\bigr)^{\overline{i+j}}
\subseteq
(V_{\mathbb{Z}\beta}^-)^i \boxtimes (V_{\mathbb{Z}\beta}^-)^j .
\]

Moreover, by Proposition~6.5,
\[
V_{\mathbb{Z}\beta}^{-+} \cong V_{\mathbb{Z}\zeta}^-,
\qquad
V_{\mathbb{Z}\beta}^{--} \cong V_{\mathbb{Z}\zeta+\frac{\zeta}{2}}^-
\]
as $V_{L_2}^{S_4}$-modules. From \cite{ADL}, we have
\begin{equation}
V_{\mathbb{Z}\zeta}^- \boxtimes_{V_{\mathbb{Z}\zeta}^+}
V_{\mathbb{Z}\zeta+\frac{\zeta}{4}}
=
V_{\mathbb{Z}\zeta+\frac{\zeta}{4}} .
\end{equation}

By Corollary~6.6, $(V_{\mathbb{Z}\beta}^-)^+$ is contained in
$V_{\mathbb{Z}\zeta+\frac{\zeta}{4}}$. Applying the same technique as in the proof of~(12), we obtain
\[
(V_{\mathbb{Z}\beta}^-)^+ \oplus (V_{\mathbb{Z}\beta}^-)^-
\subseteq
(V_{\mathbb{Z}\beta}^-)^i \boxtimes (V_{\mathbb{Z}\beta}^-)^j .
\]
Equality follows from comparing the quantum dimensions of both sides.

  \emph{proof of (20)}: By the fusion rules for $V_{L_2}^{A_4}$, the fusion product contains one copy of
$V_{\mathbb{Z}\beta+\frac{\beta}{4}}^1$ and at least one irreducible module from
$V_{\mathbb{Z}\beta+\frac{\beta}{4}}^0$.
By the fusion rules for $V_{\mathbb{Z}\zeta}^+$-modules, there exists an
intertwining operator $I$ of type
\[
I_{V_{\mathbb{Z}\zeta}^+}\!\left(
\begin{array}{c}
V_{\mathbb{Z}\zeta+\frac{3\zeta}{8}}\\
V_{\mathbb{Z}\zeta}^- \quad V_{\mathbb{Z}\zeta+\frac{3\zeta}{8}}
\end{array}
\right).
\]
Restricting $I$, we obtain a nonzero intertwining operator of type
\[
I_{V_{L_2}^{S_4}}\!\left(
\begin{array}{c}
V_{\mathbb{Z}\zeta+\frac{3\zeta}{8}}\\
(V_{\mathbb{Z}\beta}^-)^+ \quad (V_{\mathbb{Z}\beta+\frac{\beta}{4}}^0)^-
\end{array}
\right).
\]
Therefore,
\begin{equation}
(V_{\mathbb{Z}\beta}^-)^+ \boxtimes (V_{\mathbb{Z}\beta+\frac{\beta}{4}}^0)^-
=
(V_{\mathbb{Z}\beta+\frac{\beta}{4}}^0)^-
\oplus
V_{\mathbb{Z}\beta+\frac{\beta}{4}}^1 .
\end{equation}

Replacing $V_{\mathbb{Z}\zeta+\frac{3\zeta}{8}}$ by
$V_{\mathbb{Z}\zeta+\frac{\zeta}{8}}$, we similarly obtain
\begin{equation}
(V_{\mathbb{Z}\beta}^-)^+ \boxtimes (V_{\mathbb{Z}\beta+\frac{\beta}{4}}^0)^+
=
(V_{\mathbb{Z}\beta+\frac{\beta}{4}}^0)^+
\oplus
V_{\mathbb{Z}\beta+\frac{\beta}{4}}^1 .
\end{equation}

On the other hand, consider
\[
\bigl((V_{\mathbb{Z}\beta}^-)^+ \oplus (V_{\mathbb{Z}\beta}^-)^-\bigr)
\boxtimes (V_{\mathbb{Z}\beta+\frac{\beta}{4}}^0)^j .
\]
By the fusion rules for $V_{L_2}^{A_4}$, there exists an intertwining operator
$I$ of type
\[
I_{A_4}\!\left(
\begin{array}{c}
V_{\mathbb{Z}\beta+\frac{\beta}{4}}^0\\
V_{\mathbb{Z}\beta}^- \quad V_{\mathbb{Z}\beta+\frac{\beta}{4}}^0
\end{array}
\right).
\]
Then
\[
I\bigl((V_{\mathbb{Z}\beta}^-)^+ \oplus (V_{\mathbb{Z}\beta}^-)^-, z\bigr)
\bigl((V_{\mathbb{Z}\beta+\frac{\beta}{4}}^0)^j\bigr)
=
(V_{\mathbb{Z}\beta+\frac{\beta}{4}}^0)^+
\oplus
(V_{\mathbb{Z}\beta+\frac{\beta}{4}}^0)^- .
\]
This shows that
$\bigl((V_{\mathbb{Z}\beta}^-)^+ \oplus (V_{\mathbb{Z}\beta}^-)^-\bigr)
\boxtimes (V_{\mathbb{Z}\beta+\frac{\beta}{4}}^0)^j
$ contains both $(V_{\mathbb{Z}\beta+\frac{\beta}{4}}^0)^+$ and
$(V_{\mathbb{Z}\beta+\frac{\beta}{4}}^0)^-$.
Consequently, $(V_{\mathbb{Z}\beta}^-)^+ \boxtimes (V_{\mathbb{Z}\beta+\frac{\beta}{4}}^0)^j
\quad\text{and}\quad
(V_{\mathbb{Z}\beta}^-)^- \boxtimes (V_{\mathbb{Z}\beta+\frac{\beta}{4}}^0)^j$ contain different irreducible $V_{L_2}^{S_4}$-module in
$V_{\mathbb{Z}\beta+\frac{\beta}{4}}^0$.
This proves~\textup{(20)}.

\emph{proof of (21)}:By the fusion rules for $V_{L_2}^{A_4}$ and the techniques used above, we know that $(V_{\mathbb{Z}\beta}^-)^i \boxtimes V_{\mathbb{Z}\beta+\frac{\beta}{4}}^1$
must contain $
(V_{\mathbb{Z}\beta+\frac{\beta}{4}}^0)^+
\oplus
(V_{\mathbb{Z}\beta+\frac{\beta}{4}}^0)^-
\oplus
V_{\mathbb{Z}\beta+\frac{\beta}{4}}^1 .
$ Comparing the quantum dimensions of both sides, we see that there are four
remaining dimensions to be accounted for.

By the fusion rules for $V_{L_2}^{A_4}$, we have
\[
(V_{\mathbb{Z}\beta+\frac{\beta}{4}}^0)^i \boxtimes V_{\mathbb{Z}\beta+\frac{\beta}{4}}^1
\supset
(V_{\mathbb{Z}\beta}^+)^1
\oplus
(V_{\mathbb{Z}\beta}^-)^+
\oplus
(V_{\mathbb{Z}\beta}^-)^- .
\]
This implies
\[
(V_{\mathbb{Z}\beta+\frac{\beta}{4}}^0)^i \boxtimes V_{\mathbb{Z}\beta+\frac{\beta}{4}}^1
=
(V_{\mathbb{Z}\beta}^+)^1
\oplus
(V_{\mathbb{Z}\beta}^-)^+
\oplus
(V_{\mathbb{Z}\beta}^-)^- .
\]

Since all $V_{L_2}^{S_4}$-modules are self-dual, it follows that there is exactly
one copy of $(V_{\mathbb{Z}\beta+\frac{\beta}{4}}^0)^i$, for $i=0,1$, appearing in $
(V_{\mathbb{Z}\beta}^-)^i \boxtimes V_{\mathbb{Z}\beta+\frac{\beta}{4}}^1 .
$
Combining this with the previous calculations, we conclude that
$((V_{\mathbb{Z}\beta}^+)^0)^i$ , $(V_{\mathbb{Z}\beta}^+)^1$, and $(V_{\mathbb{Z}\beta}^-)^i$ do not appear in
$(V_{\mathbb{Z}\beta}^-)^i \boxtimes V_{\mathbb{Z}\beta+\frac{\beta}{4}}^1 .$
Therefore, the only choice is $(V_{\mathbb{Z}\beta}^-)^i \boxtimes V_{\mathbb{Z}\beta+\frac{\beta}{4}}^1$ contains two copies of $V_{\mathbb{Z}\beta+\frac{\beta}{4}}^1$.

\emph{proof of (22)}: We prove $(V_{\mathbb{Z}\beta}^-)^+\boxtimes V_{\mathbb{Z}\beta+\frac{\beta}{8}}^+$ only. Others are similar. By fusion rules for $V_{L_2}^{A_4}$, we know it contains at least one copy of irreducible module from $V_{\mathbb{Z}\beta+\frac{\beta}{8}}$ and 2 copies irreducible modules from $V_{\mathbb{Z}\beta+\frac{3\beta}{8}}$. By proposition 6.5, corollary 6.6, we have $(V_{\mathbb{Z}\beta}^-)^+\subseteq V_{\mathbb{Z}\zeta}, V_{\mathbb{Z}\zeta+\frac{\zeta}{4}}, V_{\mathbb{Z}\zeta+\frac{3\zeta}{4}}$ and $V_{\mathbb{Z}\beta+\frac{\beta}{8}}^+\cong V_{\frac{\zeta}{16}+\mathbb{Z}\zeta}$. Thus (22) follows from fusion rules for $V_{\mathbb{Z}\zeta}$-modules and theorem 4.4.

\emph{proof of (23)}: This directly follows from theorem 4.4 and fusion rules for $V_{L_2}^{A_4}$.

\emph{proof of (24)}: This proof is similar to (17).

\emph{proof of (25)}:By fusion rules for $V_{L_2}^{A_4}$, $(V_{\mathbb{Z}\beta+\frac{\beta}{4}}^0)^i\boxtimes (V_{\mathbb{Z}\beta+\frac{\beta}{4}}^0)^j$ must contain at least one module from $V_{\mathbb{Z}\beta}^-$ and $(V_{\mathbb{Z}\beta}^+)^0$. By self-duality, it is clear $((V_{\mathbb{Z}\beta}^+)^0)^+$ is in the tensor product  if $i=j$ and $((V_{\mathbb{Z}\beta}^+)^0)^-$ is in the tensor product if $i\neq j$. By (20), it is clear $(V_{\mathbb{Z}\beta}^-)^+\subseteq (V_{\mathbb{Z}\beta+\frac{\beta}{4}}^0)^i\boxtimes (V_{\mathbb{Z}\beta+\frac{\beta}{4}}^0)^j$ if $i=j$ and $(V_{\mathbb{Z}\beta}^-)^-\subseteq (V_{\mathbb{Z}\beta+\frac{\beta}{4}}^0)^i\boxtimes (V_{\mathbb{Z}\beta+\frac{\beta}{4}}^0)^j$ if $i\neq j$.

\emph{proof of (26)}: this is already done in proof of (21).

\emph{proof of (27)}:This proof is similar to (17).

\emph{proof of (28)}: This follows from fusion rules for $V_{L_2}^{A_4}$ and theorem 4.4.

\emph{proof of (29)}:This proof is similar to (17).

\emph{proof of (30)}:This proof is similar to (17).

\emph{proof of (31)}:By the fusion rules for $V_{L_2}^{A_4}$ and the techniques used above, the fusion
product contains
\[
((V_{\mathbb{Z}\beta}^+)^0)^+,\quad
((V_{\mathbb{Z}\beta}^+)^0)^-,\quad
(V_{\mathbb{Z}\beta}^-)^+,\quad
(V_{\mathbb{Z}\beta}^-)^- .
\]
By (7.28), the tensor product contains exactly one copy of
$(V_{\mathbb{Z}\beta}^+)^1$.
By (7.40), there are two copies of $(V_{\mathbb{Z}\beta}^-)^i$ for $i=0,1$.
The equality then follows by comparing the quantum dimensions of both sides.

\emph{proof of (32)}: This follows from fusion rules for $V_{L_2}^{A_4}$ and technique used in (12).

Proof of (33) will be given later.

\emph{proof of (34)}:This proof is similar to (17).

\emph{proof of (35)}: The quantum dimension of 
$
(V_{\mathbb{Z}\beta+\tfrac{i\beta}{8}})^j \boxtimes (V_{\mathbb{Z}\beta+\tfrac{k\beta}{8}})^l
$
is equal to $36$.  

For $i=k$, the fusion rules for $V_{L_2}^{A_4}$ give
\[
\begin{aligned}
V_{\mathbb{Z}\beta+\tfrac{i\beta}{8}}\boxtimes_{A_4} V_{\mathbb{Z}\beta+\tfrac{i\beta}{8}}
&=
V_{\mathbb{Z}\beta+\tfrac{\beta}{4}}^0
\oplus V_{\mathbb{Z}\beta+\tfrac{\beta}{4}}^1
\oplus V_{\mathbb{Z}\beta+\tfrac{\beta}{4}}^2
\oplus (V_{\mathbb{Z}\beta}^+)^0
\oplus (V_{\mathbb{Z}\beta}^+)^1
\oplus (V_{\mathbb{Z}\beta}^+)^2 \\
&\quad\oplus V_{\mathbb{Z}\beta}^-
\oplus 2V_{\mathbb{Z}\beta+\tfrac{\beta}{8}}
\oplus 2V_{\mathbb{Z}\beta+\tfrac{3\beta}{8}},
\end{aligned}
\]
and
\[
\begin{aligned}
V_{\mathbb{Z}\beta+\tfrac{\beta}{8}}\boxtimes_{A_4} V_{\mathbb{Z}\beta+\tfrac{3\beta}{8}}
&=
V_{\mathbb{Z}\beta+\tfrac{\beta}{4}}^0
\oplus V_{\mathbb{Z}\beta+\tfrac{\beta}{4}}^1
\oplus V_{\mathbb{Z}\beta+\tfrac{\beta}{4}}^2
\oplus 2V_{\mathbb{Z}\beta}^- \\
&\quad\oplus 2V_{\mathbb{Z}\beta+\tfrac{\beta}{8}}
\oplus 2V_{\mathbb{Z}\beta+\tfrac{3\beta}{8}}.
\end{aligned}
\]

The restriction of quantum dimensions implies that there is exactly one irreducible
$V_{L_2}^{S_4}$-module arising from each of the following
$V_{L_2}^{A_4}$-modules if $i=k$:
\[
V_{\mathbb{Z}\beta+\tfrac{\beta}{4}}^0,\;
V_{\mathbb{Z}\beta+\tfrac{\beta}{4}}^1,\;
(V_{\mathbb{Z}\beta}^+)^0,\;
(V_{\mathbb{Z}\beta}^+)^1,\;
V_{\mathbb{Z}\beta}^-,
\]
and two irreducible $V_{L_2}^{S_4}$-modules arising from each of
$V_{\mathbb{Z}\beta+\tfrac{\beta}{8}}$ and $V_{\mathbb{Z}\beta+\tfrac{3\beta}{8}}$.

To determine whether $(V_{\mathbb{Z}\beta+\tfrac{\beta}{4}}^0)^+$ or
$(V_{\mathbb{Z}\beta+\tfrac{\beta}{4}}^0)^-$ appears in the tensor product,
we use our calculation in (27).
We only compute
$
V_{\mathbb{Z}\beta+\tfrac{\beta}{8}}^+ \boxtimes V_{\mathbb{Z}\beta+\tfrac{\beta}{8}}^+.
$
Since
\[
(V_{\mathbb{Z}\beta+\tfrac{\beta}{4}}^0)^+ \boxtimes (V_{\mathbb{Z}\beta+\tfrac{\beta}{8}})^+
=
V_{\mathbb{Z}\beta+\tfrac{3\beta}{8}}^+ \oplus V_{\mathbb{Z}\beta+\tfrac{\beta}{8}}^+,
\]
by self-duality we conclude that
$
(V_{\mathbb{Z}\beta+\tfrac{\beta}{8}})^+ \boxtimes (V_{\mathbb{Z}\beta+\tfrac{\beta}{8}})^+
$
contains $(V_{\mathbb{Z}\beta+\tfrac{\beta}{4}}^0)^+$.  
The other cases are similar.

Applying the same technique to (22), we also see that the tensor product contains
$(V_{\mathbb{Z}\beta}^-)^{\overline{j+l}}$ if $i=k$, and
$(V_{\mathbb{Z}\beta}^-)^+ \oplus (V_{\mathbb{Z}\beta}^-)^-$ otherwise.

By the previous arguments, it is clear that
\[
((V_{\mathbb{Z}\beta}^+)^0)^{\overline{j+l}},\quad
V_{\mathbb{Z}\beta+\tfrac{\beta}{4}}^1,\quad
(V_{\mathbb{Z}\beta}^+)^1
\]
all appear in the tensor product.

It remains to determine the summands related to
$V_{\mathbb{Z}\beta+\tfrac{\beta}{8}}$ and
$V_{\mathbb{Z}\beta+\tfrac{3\beta}{8}}$.
By \cite{ADL}, there exists an intertwining operator
$
I \in I_{V_{\mathbb{Z}\zeta}^+}\!\left(
\begin{array}{c}
V_{\mathbb{Z}\zeta}^{T_1,+} \\
V_{\mathbb{Z}\zeta+\tfrac{\zeta}{16}} \quad V_{\mathbb{Z}\zeta}^{T_1,+}
\end{array}
\right).
$
Using the technique from (12), we obtain that
$
V_{\mathbb{Z}\beta+\tfrac{\beta}{8}}^+ \boxtimes V_{\mathbb{Z}\beta+\tfrac{\beta}{8}}^+
$
contains
$V_{\mathbb{Z}\beta+\tfrac{\beta}{8}}^+ \oplus V_{\mathbb{Z}\beta+\tfrac{\beta}{8}}^-$,
and
$
V_{\mathbb{Z}\beta+\tfrac{\beta}{8}}^+ \boxtimes V_{\mathbb{Z}\beta+\tfrac{\beta}{8}}^-
$
also contains
$V_{\mathbb{Z}\beta+\tfrac{\beta}{8}}^+ \oplus V_{\mathbb{Z}\beta+\tfrac{\beta}{8}}^-$.
The remaining cases are similar.

Combining all these results, we obtain $(7.55)-(7.57)$.

\emph{proof of (36)}: (\ref{7.57}-\ref{7.60}) follows directly from the S-matrix and previous calculations. We prove (\ref{7.61}-\ref{7.69}). The summands consisting of modules of the form
$V_{\mathbb{Z}\gamma+\frac{i\gamma}{18}}$
are determined by the $S$-matrix.
 By fusion rules for $V_{L_2}^{A_4}$, we know the tensor product must contain at least one copy of $V_{\mathbb{Z}\beta+\frac{\beta}{4}}^1,V_{\mathbb{Z}\beta+\frac{\beta}{8}}^{\pm},$ and $V_{\mathbb{Z}\beta+\frac{3\beta}{8}}^{\pm}$. After counting quantum dimensions, a quantum dimension of 4 remains. From fusion rules for $V_{L_2}^{A_4}$, the tensor product contains $(V_{\mathbb{Z}\beta+\frac{\beta}{4}}^0)^+,(V_{\mathbb{Z}\beta+\frac{\beta}{4}}^0)^-$ unless $(k=0,l=2), (k=l=1), (k=2,l=0)$. For these three cases, we know $((V_{\mathbb{Z}\beta}^+)^0)^{\pm},(V_{\mathbb{Z}\beta}^+)^1, (V_{\mathbb{Z}\beta}^-)^{\pm},(V_{\mathbb{Z}\beta+\frac{\beta}{4}}^0)^{\pm}$ are not in the tensor product from those fusion rules that are already calculated. Thus, we have no other choice.

\emph{proof of (33)}: Use Proposition 6.7, we know $V_{\mathbb{Z}\beta+\frac{\beta}{4}}^1\boxtimes V_{\mathbb{Z}\gamma+\frac{r\gamma}{18}}$ must contain $V_{\mathbb{Z}\gamma+\frac{2\gamma}{18}}\oplus V_{\mathbb{Z}\gamma+\frac{4\gamma}{18}} \oplus V_{\mathbb{Z}\gamma+\frac{8\gamma}{18}}$ if $r$ is odd and $V_{\mathbb{Z}\beta+\frac{\beta}{4}}^1\boxtimes V_{\mathbb{Z}\gamma+\frac{r\gamma}{18}}$ must contain $V_{\mathbb{Z}\gamma+\frac{\gamma}{18}}\oplus V_{\mathbb{Z}\gamma+\frac{5\gamma}{18}} \oplus V_{\mathbb{Z}\gamma+\frac{7\gamma}{18}}$ if $r$ is even. By (36) and self-duality, we can get (33). Equality follows from counting quantum dimensions.

\end{proof}

For $1 \le i,j \le 15$ with $i ,j \neq 0 \pmod{2}$, we divide the tensor product
$
V_{\mathbb{Z}\zeta+\frac{i\zeta}{32}} \boxtimes V_{\mathbb{Z}\zeta+\frac{j\zeta}{32}}
$
into three parts.

We call all irreducible modules belonging to $M_{12}$--$M_{17}$ \emph{Part A}, and all irreducible modules belonging to $M_{8}$--$M_{11}$ \emph{Part B}. All remaining irreducible modules are said to belong to \emph{Part C}.

This decomposition will be convenient for describing the fusion rules for
$
V_{\mathbb{Z}\zeta+\frac{s\zeta}{32}} \boxtimes V_{\mathbb{Z}\zeta+\frac{s\zeta}{32}}.
$
For brevity, we will use $(i,j)$ to denote the fusion product $M_i \boxtimes M_j$. The full fusion rules are obtained by adding \emph{Part A}, \emph{Part B}, and \emph{Part C}.

\begin{theorem}

    Part A is as follows: 
    \begin{eqnarray}
        & &(18,18)=(18,21)=(18,22)=(18,25)=(19,19)=(19,20)=(19,23)=(19,24)\nonumber\\
        & &(20,20)=(20,23)=(20,24)=(21,21)=(21,22)=(21,25)=(22,22)=(22,25)\nonumber\\
        & &(23,23)=(23,24)=(24,24)=(25,25)=M_{12}\oplus M_{15}\oplus M_{16}.
    \end{eqnarray}
    Other pairs have part A= $M_{13}\oplus M_{14}\oplus M_{17}$.

    Here is part B:
    \begin{eqnarray}
       & & (18,18)=(18,19)=(19,20)=(20,21)=(21,22)=(22,23)=(23,24)=(24,25)=(25,25)\nonumber\\
        & &=M_8.\nonumber\\
        & &(18,24)=(18,25)=(19,23)=(19,25)=(20,22)=(20,24)=(21,21)=(21,23)=(22,22)\nonumber\\
        & &=M_9.\nonumber\\
        & &(18,20)=(18,21)=(19,19)=(19,22)=(20,23)=(21,24)=(22,25)=(23,25)=(24,24)\nonumber\\
        & &=M_{10}.\nonumber\\
        & &(18,22)=(18,23)=(19,21)=(19,24)=(20,25)=(20,20)=(21,25)=(22,24)=(23,23)\nonumber\\
        & &=M_{11}.\nonumber\\
    \end{eqnarray}

Here is part C:
\begin{eqnarray}
    & &(18,18)=(19,19)=(20,20)=(21,21)=(22,22)=(23,23)=(24,24)=(25,25)\nonumber\\
    & &=((V_{\mathbb{Z}\beta}^+)^0)^+\oplus (V_{\mathbb{Z}\beta}^+)^1
\oplus (V_{\mathbb{Z}\beta}^-)^+.\nonumber\\
& &(18,19)=(18,20)=(19,21)=(20,22)=(21,23)=(22,24)=(23,25)=(24,25)\nonumber\\
& &=(V_{\mathbb{Z}\beta+\frac{\beta}{4}}^0)^+\oplus V_{\mathbb{Z}\beta+\frac{\beta}{4}}^1.\nonumber\\
& &(18,21)=(18,22)=(19,20)=(19,23)=(20,24)=(21,25)=(22,25)=(23,24)\nonumber\\
& &=(V_{\mathbb{Z}\beta}^-)^+\oplus (V_{\mathbb{Z}\beta}^-)^-.\nonumber\\
& &(18,23)=(18,24)=(19,22)=(19,25)=(20,21)=(20,25)=(21,24)=(22,23)\nonumber\\
& &=(V_{\mathbb{Z}\beta+\frac{\beta}{4}}^0)^-\oplus V_{\mathbb{Z}\beta+\frac{\beta}{4}}^1.\nonumber\\
& &(18,25)=(19,24)=(20,23)=(21,22)=((V_{\mathbb{Z}\beta}^+)^0)^-\oplus (V_{\mathbb{Z}\beta}^+)^1
\oplus (V_{\mathbb{Z}\beta}^-)^-.\nonumber\\
\end{eqnarray} 
\end{theorem}
\begin{proof}
    Part A and Part B directly follows from S matrix. Part C follows from self-duality and calculations in Theorem 7.2.
\end{proof}
\begin{theorem}
    \begin{equation}
             (1).V_{\mathbb{Z}\beta+\frac{\beta}{4}}^1\boxtimes V_{\mathbb{Z}\zeta}^{T_2,i}=2V_{\mathbb{Z}\zeta}^{T_2,+}\oplus 2V_{\mathbb{Z}\zeta}^{T_2,-}.\label{7.54}(i=0,1)
\end{equation}
    \begin{eqnarray}
        (2).& &V_{\mathbb{Z}\zeta}^{T_2,i}\boxtimes V_{\mathbb{Z}\zeta}^{T_2,j}=(\bigoplus_{k=1,k\neq 0\mbox { mod }3}^{8}2V_{\mathbb{Z}\gamma+\frac{k\gamma}{18}})\oplus (\bigoplus_{s=1,3,t=0,1}V_{\mathbb{Z}\beta+\frac{s\beta}{8}}^t)\oplus (V_{\mathbb{Z}\beta+\frac{\beta}{4}}^0)^+\oplus (V_{\mathbb{Z}\beta+\frac{\beta}{4}}^0)^-\nonumber\\
        & &\oplus ((V_{\mathbb{Z}\beta}^+)^0)^{\overline{i+j}}
        \oplus (V_{\mathbb{Z}\beta}^+)^1\oplus 2V_{\mathbb{Z}\beta+\frac{\beta}{4}}^1\oplus 2(V_{\mathbb{Z}\beta}^-)^{\overline{i+j+1}}\oplus (V_{\mathbb{Z}\beta}^-)^{\overline{i+j}}.(i,j=0,1)
    \end{eqnarray}

\end{theorem}

\begin{proof}
    Let us prove (2) first. $V_{\mathbb{Z}\zeta}^{T_2,i}\boxtimes V_{\mathbb{Z}\zeta}^{T_2,j}$ contains 
    \begin{align}
        (\bigoplus_{k=1,k\neq 0\mbox { mod }3}^{8}2V_{\mathbb{Z}\gamma+\frac{k\gamma}{18}})\oplus (\bigoplus_{s=1,3,t=0,1}V_{\mathbb{Z}\beta+\frac{s\beta}{8}}^t)
    \end{align}
    by S-matrix. By Theorem 7.1, Theorem 7.2, $V_{\mathbb{Z}\zeta}^{T_2,i}\boxtimes V_{\mathbb{Z}\zeta}^{T_2,j}$ contains exactly one copy of $(V_{\mathbb{Z}\beta+\frac{\beta}{4}}^0)^+,(V_{\mathbb{Z}\beta+\frac{\beta}{4}}^0)^-,((V_{\mathbb{Z}\beta}^+)^0)^{\overline{i+j}},(V_{\mathbb{Z}\beta}^+)^1,2(V_{\mathbb{Z}\beta}^-)^{\overline{i+j+1}},(V_{\mathbb{Z}\beta}^-)^{\overline{i+j}}$. By counting quantum dimensions, a quantum dimension of 8 remains. So, the only remaining summand is $2V_{\mathbb{Z}\beta+\frac{\beta}{4}}^1$. This proves (2).

    (1) follows from (2).
\end{proof}

\section{Appendix}
\def\theequation{6.\arabic{equation}}
\setcounter{equation}{0}

For convenience, in the following diagram, we will use $W_j$ to denote $\frac{4}{3}(e^{\frac{j\pi i }{9}}+e^{-\frac{j\pi i }{9}})$ for $j=1,2,4,5,7,8$ and $T_j$ to denote $e^{\frac{j\pi i}{8}}+e^{\frac{-j\pi i}{8}}$ for $j= 1,3,5,7$ and $P_j$ to denote $e^{\frac{j\pi i}{16}}+e^{\frac{-j\pi i}{16}}$ for $j= 1,3,5,7,9,11,13,15$. The following is the part of $S$-matrix for irreducible $V_{L_{2}}^{S_{4}}$-modules
that we need:

\begin{tabular}{|c|c|c|c|c|c|c|c|c|c|c|c|c|}
\hline
$\sqrt{32}S_{i,j}$ & 0 & 8 & 9 & 10 & 11 & 12 & 13 & 14 & 15 & 16 & 17\tabularnewline
\hline
\hline
0 & $\frac{1}{6}$ & 1 & 1 & 1 & 1 & $\frac{4}{3}$ &  $\frac{4}{3}$ & $\frac{4}{3}$ & $\frac{4}{3}$ & $\frac{4}{3}$ & $\frac{4}{3}$ \tabularnewline
\hline
1 & $\frac{1}{6}$ & 1 & 1 & 1 & 1 & $\frac{4}{3}$ &  $\frac{4}{3}$ & $\frac{4}{3}$ & $\frac{4}{3}$ & $\frac{4}{3}$ & $\frac{4}{3}$ \tabularnewline
\hline
2 & $\frac{1}{3}$ & 2 & 2 & 2 & 2 & $-\frac{4}{3}$ &  $-\frac{4}{3}$ & $-\frac{4}{3}$ & $-\frac{4}{3}$ & $-\frac{4}{3}$ & $-\frac{4}{3}$ \tabularnewline
\hline
3 & $\frac{1}{2}$ & -1 & -1 & -1 & -1 & $0$ &  $0$ & $0$ & $0$ & $0$ & $0$ \tabularnewline
\hline
4 & $\frac{1}{2}$ & -1 & -1 & -1 & -1 & $0$ &  $0$ & $0$ & $0$ & $0$ & $0$ \tabularnewline
\hline
5 & $\frac{1}{3}$ & 0 & 0 & 0 & 0 & $\frac{4}{3}$ &  $-\frac{4}{3}$ & $-\frac{4}{3}$ & $\frac{4}{3}$ & $\frac{4}{3}$ & $-\frac{4}{3}$ \tabularnewline
\hline
6 & $\frac{1}{3}$ & 0 & 0 & 0 & 0 & $\frac{4}{3}$ &  $-\frac{4}{3}$ & $-\frac{4}{3}$ & $\frac{4}{3}$ & $\frac{4}{3}$ & $-\frac{4}{3}$ \tabularnewline
\hline
7 & $\frac{2}{3}$ & 0 & 0 & 0 & 0 & $-\frac{4}{3}$ &  $\frac{4}{3}$ & $\frac{4}{3}$ & $-\frac{4}{3}$ & $-\frac{4}{3}$ & $\frac{4}{3}$ \tabularnewline
\hline
8 & $1$ & $\sqrt{2}$ & $\sqrt{2}$ & $-\sqrt{2}$ & $-\sqrt{2}$ & $0$ &  $0$ & $0$ & $0$ & $0$ & $0$ \tabularnewline
\hline
9 & $1$ & $\sqrt{2}$ & $\sqrt{2}$ & $-\sqrt{2}$ & $-\sqrt{2}$ & $0$ &  $0$ & $0$ & $0$ & $0$ & $0$ \tabularnewline
\hline
10 & $1$ & $-\sqrt{2}$ & $-\sqrt{2}$ & $\sqrt{2}$ & $\sqrt{2}$ & $0$ &  $0$ & $0$ & $0$ & $0$ & $0$ \tabularnewline
\hline
11 & $1$ & $-\sqrt{2}$ & $-\sqrt{2}$ & $\sqrt{2}$ & $\sqrt{2}$ & $0$ &  $0$ & $0$ & $0$ & $0$ & $0$ \tabularnewline
\hline
12 & $\frac{4}{3}$ & 0 & 0 & 0 & 0 & $W_1$ &  $W_2$ & $W_4$ & $W_5$ & $W_7$ & $W_8$ \tabularnewline
\hline
13 & $\frac{4}{3}$ & 0 & 0 & 0 & 0 & $W_2$ &  $W_4$ & $W_8$ & $W_8$ & $W_4$ & $W_2$ \tabularnewline
\hline
14 & $\frac{4}{3}$ & 0 & 0 & 0 & 0 & $W_4$ &  $W_8$ & $W_2$ & $W_2$ & $W_8$ & $W_4$ \tabularnewline
\hline
15 & $\frac{4}{3}$ & 0 & 0 & 0 & 0 & $W_5$ &  $W_8$ & $W_2$ & $W_7$ & $W_1$ & $W_4$ \tabularnewline
\hline
16 & $\frac{4}{3}$ & 0 & 0 & 0 & 0 & $W_7$ &  $W_4$ & $W_8$ & $W_1$ & $W_5$ & $W_2$ \tabularnewline
\hline
17 & $\frac{4}{3}$ & 0 & 0 & 0 & 0 & $W_8$ &  $W_2$ & $W_4$ & $W_4$ & $W_2$ & $W_8$ \tabularnewline
\hline
18 & $1$ & $T_1$ & $T_7$ & $T_3$ & $T_5$ & $0$ &  $0$ & $0$ & $0$ & $0$ & $0$ \tabularnewline
\hline
19 & $1$ & $T_3$ & $T_5$ & $T_7$ & $T_1$ & $0$ &  $0$ & $0$ & $0$ & $0$ & $0$ \tabularnewline
\hline
20 & $1$ & $T_5$ & $T_3$ & $T_1$ & $T_7$ & $0$ &  $0$ & $0$ & $0$ & $0$ & $0$ \tabularnewline
\hline
21 & $1$ & $T_7$ & $T_1$ & $T_5$ & $T_3$ & $0$ &  $0$ & $0$ & $0$ & $0$ & $0$ \tabularnewline
\hline
22 & $1$ & $T_7$ & $T_1$ & $T_5$ & $T_3$ & $0$ &  $0$ & $0$ & $0$ & $0$ & $0$ \tabularnewline
\hline
23 & $1$ & $T_5$ & $T_3$ & $T_1$ & $T_7$ & $0$ &  $0$ & $0$ & $0$ & $0$ & $0$ \tabularnewline
\hline
24 & $1$ & $T_3$ & $T_5$ & $T_7$ & $T_1$ & $0$ &  $0$ & $0$ & $0$ & $0$ & $0$ \tabularnewline
\hline
25 & $1$ & $T_1$ & $T_7$ & $T_3$ & $T_5$ & $0$ &  $0$ & $0$ & $0$ & $0$ & $0$ \tabularnewline
\hline
26 & $2$ & $0$ & $0$ & $0$ & $0$ & $0$ &  $0$ & $0$ & $0$ & $0$ & $0$ \tabularnewline
\hline
27 & $2$ & $0$ & $0$ & $0$ & $0$ & $0$ &  $0$ & $0$ & $0$ & $0$ & $0$ \tabularnewline
\hline
\end{tabular}\\

\begin{tabular}{|c|c|c|c|c|c|c|c|c|c|c|}
\hline
$\sqrt{32}S_{i,j}$ & 18 & 19 & 20 & 21 & 22 & 23 & 24 & 25 & 26 & 27\tabularnewline
\hline
\hline
0 & 1 & 1 & 1 & 1 & 1 & 1 & 1 & 1 & 2 & 2\tabularnewline
\hline
1 & -1 & -1 & -1 & -1 & -1 & -1 & -1 & -1 & -2 & -2\tabularnewline
\hline
2 & 0 & 0 & 0 & 0 & 0 & 0 & 0 & 0 & 0 & 0\tabularnewline
\hline
3 & 1 & 1 & 1 & 1 & 1 & 1 & 1 & 1 & -2 & -2\tabularnewline
\hline
4 & -1 & -1 & -1 & -1 & -1 & -1 & -1 & -1 & 2 & 2\tabularnewline
\hline
5 & $\sqrt{2}$ & -$\sqrt{2}$ & -$\sqrt{2}$ & $\sqrt{2}$ & $\sqrt{2}$ & -$\sqrt{2}$ & -$\sqrt{2}$ & $\sqrt{2}$ & 0 & 0\tabularnewline
\hline
6 & $-\sqrt{2}$ & $\sqrt{2}$ & $\sqrt{2}$ & $-\sqrt{2}$ & $-\sqrt{2}$ & $\sqrt{2}$ & $\sqrt{2}$ & $-\sqrt{2}$ & 0 & 0\tabularnewline
\hline
7 & 0 & 0 & 0 & 0 & 0 & 0 & 0 & 0 & 0 & 0\tabularnewline
\hline
8 & $T_1$ & $T_3$ & $T_5$ & $T_7$ & $T_7$ & $T_5$ & $T_3$ & $T_1$ & 0 & 0\tabularnewline
\hline
9 & $T_7$ & $T_5$ & $T_3$ & $T_1$ & $T_1$ & $T_3$ & $T_5$ & $T_7$ & 0 & 0\tabularnewline
\hline
10 & $T_3$ & $T_7$ & $T_1$ & $T_5$ & $T_5$ & $T_1$ & $T_7$ & $T_3$ & 0 & 0\tabularnewline
\hline
11 & $T_5$ & $T_1$ & $T_7$ & $T_3$ & $T_3$ & $T_7$ & $T_1$ & $T_5$ & 0 & 0\tabularnewline
\hline
12 & 0 & 0 & 0 & 0 & 0 & 0 & 0 & 0 & 0 & 0\tabularnewline
\hline
13 & 0 & 0 & 0 & 0 & 0 & 0 & 0 & 0 & 0 & 0\tabularnewline
\hline
14 & 0 & 0 & 0 & 0 & 0 & 0 & 0 & 0 & 0 & 0\tabularnewline
\hline
15 & 0 & 0 & 0 & 0 & 0 & 0 & 0 & 0 & 0 & 0\tabularnewline
\hline
16 & 0 & 0 & 0 & 0 & 0 & 0 & 0 & 0 & 0 & 0\tabularnewline
\hline
17 & 0 & 0 & 0 & 0 & 0 & 0 & 0 & 0 & 0 & 0\tabularnewline
\hline
18 & $P_1$ & $P_3$ & $P_5$ & $P_7$ & $P_9$ & $P_{11}$ & $P_{13}$ & $P_{15}$ & 0 & 0\tabularnewline
\hline
19 & $P_3$ & $P_9$ & $P_{15}$ & $P_{11}$ & $P_5$ & $P_{1}$ & $P_{7}$ & $P_{13}$ & 0 & 0\tabularnewline
\hline
20 & $P_5$ & $P_{15}$ & $P_{7}$ & $P_{3}$ & $P_{13}$ & $P_{9}$ & $P_{1}$ & $P_{11}$ & 0 & 0\tabularnewline
\hline
21 & $P_7$ & $P_{11}$ & $P_{3}$ & $P_{15}$ & $P_{1}$ & $P_{13}$ & $P_{5}$ & $P_{9}$ & 0 & 0\tabularnewline
\hline
22 & $P_9$ & $P_{5}$ & $P_{13}$ & $P_{1}$ & $P_{15}$ & $P_{3}$ & $P_{11}$ & $P_{7}$ & 0 & 0\tabularnewline
\hline
23 & $P_{11}$ & $P_{1}$ & $P_{9}$ & $P_{13}$ & $P_{3}$ & $P_{7}$ & $P_{15}$ & $P_{5}$ & 0 & 0\tabularnewline
\hline
24 & $P_{13}$ & $P_{7}$ & $P_{1}$ & $P_{5}$ & $P_{11}$ & $P_{15}$ & $P_{9}$ & $P_{3}$ & 0 & 0\tabularnewline
\hline
25 & $P_{15}$ & $P_{13}$ & $P_{11}$ & $P_{9}$ & $P_{7}$ & $P_{5}$ & $P_{3}$ & $P_{1}$ & 0 & 0\tabularnewline
\hline
26 & 0 & 0 & 0 & 0 & 0 & 0 & 0 & 0 & $2\sqrt{2}$ & $-2\sqrt{2}$\tabularnewline
\hline
27 & 0 & 0 & 0 & 0 & 0 & 0 & 0 & 0 & $-2\sqrt{2}$ & $2\sqrt{2}$\tabularnewline
\hline
\end{tabular}


\begin{thebibliography} {DLM1}
    \bibitem[A]{A} T. Abe, Fusion rules for the charge conjugation
orbifold, {\em J. Algebra} \emph{ }\textbf{242 }(2001) 624-655.
    \bibitem[ABD]{ABD} T. Abe, G. Buhl, and C. Dong, Rationality, regularity, and $C_2$-cofiniteness, {\em Trans. Amer. Math. Soc} \textbf{356} (2004), 3391-3402.
    \bibitem[ADL]{ADL} T. Abe, C. Dong and H. Li, Fusion rules for the
vertex operator $M(1)^{+}$ and $V_{L}^{+}$, \emph{Comm. Math. Phys.}\emph{.}
\textbf{253} (2005) 171-219.
    \bibitem[ADJR]{ADJR} C,Ai, C. Dong, X. Jiao and L, Ren, The irreducible modules and fusion rules for the parafermion vertex operator algebras, {\em Trans. Amer. Math. Soc.} \textbf{370} (2018) 5963-5981.
    \bibitem[DG]{DG} C. Dong and R. Griess, Rank one lattice type vertex
operator algebras and their automorphism groups, { J. Algebra} \emph{
}\textbf{208 }(1998) 262-275.
    \bibitem[DJ]{DJ} C. Dong, C. Jiang, Representations of the vertex operator algebra $V_{L_2}^{A_4}$, {\em J. Algebra} \textbf{377} (2013) 76-96.
    \bibitem[DJJJY]{DJJJY} C. Dong, C. Jiang, Q. Jiang, X. Jiao, N. Yu, Fusion rules for the vertex operator algebra $V_{L_2}^{A_4}$, {\em J. Algebra} \textbf{423} (2015) 476-505.
    \bibitem[DJX]{DJX} C. Dong, X. Jiao and F. Xu, Quantum dimensions
and Quantum Galois theory, {\em Trans. AMS.} {\bf 365} (2013), 6441-6469.
    \bibitem[DL]{DL} C. Dong, J. Lepowsky. Generalized vertex algebras and relative vertex operators, {\em Progress in Math.} \textbf{112}, 1993.
    \bibitem[DLM1]{DLM1} C. Dong, H. Li and G. Mason, Twisted representations
of vertex operator algebras, {\em Math. Ann.} \textbf{310} (1998) 571-600.
    \bibitem[DLM2]{DLM2} C. Dong, H. Li and G. Mason, Modular-Invariance
of Trace Functions in Orbifold Theory and Generalized Moonshine, {\em Comm.
Math. Phys.} \textbf{214} (2000) 1-56.
    \bibitem[DM1]{DM1} C. Dong and G. Mason, On quantum Galois theory,
{\em Duke Math. J.} \emph{ }\textbf{86} (1997) 305-321.

    \bibitem[DN]{DN} C. Dong and K. Nagatomo, Representations of vertex
operator algebra $V_{L}^{+}$ for rank one lattice $L$, {\em Comm. Math.
Phys.} \textbf{202} (1999) 169-195.
    \bibitem[DRX]{DRX} C. Dong, L. Ren, F. Xu. On orbifold Theory, {\em Adv. Math.} \textbf{321} (2017) 1-30.
    \bibitem[FHL]{FHL} I. B. Frenkel, Y. Huang and J. Lepowsky, On axiomatic
approaches to vertex operator algebras and modules, {\em Memoirs American
Math. Soc.} \textbf{104}, 1993.
    \bibitem[FLM]{FLM} I. B. Frenkel, J. Lepowsky and A. Meurman, Vertex
operator algebras and the monster, {\em Pure and Applied Math.} Vol. 134,
Academic Press, Massachusetts, 1988.

    \bibitem[H]{H} Y.-Z. Huang, Vertex operator algebras and the Verlinde
Conjecture, {\em Comm. Contemp. Math.} \textbf{10} (2008) 103-154.
    \bibitem[HKL]{HKL}Y. Huang, A. Kirillov Jr. and J. Lepowsky, Braided tensor categories and extensions of vertex operator algebras, {\em Comm. Math. Phys}. \textbf{337} (2015), 1143-1159.
    \bibitem[K]{K} Kiritsis, E, Proof of the completeness of the classification of rational conformal field theories with $c=1$, {\em Phys. Lett.} \textbf{B217},427-430(1989).
    \bibitem[S]{S}J.-P. Serre, A Course in Arithmetic, Graduate Texts
in Mathematics\emph{,} vol. 7, Springer-Verlag (1973)
    \bibitem[V]{V} E. Verlinde, Fusion rules and modular transformation
in 2D conformal field theory, {\em Nucl. Phys.} \textbf{B300} (1988), 360-376.
    \bibitem[WZ]{WZ} L. Wu, L. Zhang. Representations of Vertex Operator Algebra $V_{L_2}^{S_4}$, $V_{L_2}^{A_5}$, and Their Quantum Dimensions. {\em 	arXiv:1603.01797}, Mar 2016.
    \bibitem[Z]{Z} Y. Zhu, Modular invariance of characters of vertex
operator algebras, {\em J. Amer. Math. Soc.} \textbf{9 }(1996) 237-302.

\end{thebibliography}
\end {document}